\documentclass{amsart}

\usepackage{amsmath,amssymb,verbatim}
\usepackage{amsthm}
\usepackage{enumerate}
\usepackage{url}
\usepackage{mathtools}

\newtheorem{theorem}{Theorem}[section]
\newtheorem{lemma}[theorem]{Lemma}
\newtheorem{corollary}[theorem]{Corollary}
\newtheorem{proposition}[theorem]{Proposition}
\newtheorem{fact}[theorem]{Fact}

\theoremstyle{definition}
\newtheorem{definition}[theorem]{Definition}

\newtheorem{remark}[theorem]{Remark}
\newtheorem{notation}[theorem]{Notation}

\def\seq{\subseteq}
\def\nv{\text{-}}
\def\inv{^{\text{-}1}}
\def\st{\textnormal{st}}
\def\smd{\raisebox{.4pt}{\textrm{\scriptsize{~\!$\triangle$\!~}}}}

\def\cU{\mathcal{U}}
\def\cB{\mathcal{B}}

\def\cF{\mathcal{F}}

\def\cP{\mathcal{P}}
\def\cR{\mathcal{R}}
\def\cS{\mathcal{S}}

\def\cZ{\mathcal{Z}}
\def\N{\mathbb{N}}
\def\G{\mathbb{G}}
\def\B{\mathbb{B}}
\def\E{\mathbb{E}}
\def\F{\mathbb{F}}
\def\R{\mathbb{R}}

\def\Z{\mathbb{Z}}

\def\Stab{\operatorname{Stab}}

\def\Th{\operatorname{Th}}
\def\VC{\operatorname{VC}}
\def\Def{\operatorname{Def}}

\newcommand{\miff}{\makebox[.4in]{$\Leftrightarrow$}}
\newcommand{\claim}{\hfill$\dashv_{\text{\scriptsize{claim}}}$}
\newcommand{\abar}{\bar{a}}
\newcommand{\bbar}{\bar{b}}

\newcommand{\gbar}{\bar{g}}
\renewcommand{\hbar}{\bar{h}}
\newcommand{\ubar}{\bar{u}}

\newcommand{\xbar}{\bar{x}}
\newcommand{\ybar}{\bar{y}}

\AtBeginDocument{%
   \def\MR#1{}
}

\title{Approximate subgroups with bounded VC-dimension}

\date{May 21, 2020}

\author[G. Conant]{Gabriel Conant}

\address{Department of Pure Mathematics and Mathematical Statistics\\
University of Cambridge\\
Cambridge CB3 0WB\\
 UK}
\email{gconant@maths.cam.ac.uk}

\author[A. Pillay]{Anand Pillay}

\thanks{Partially supported by NSF grants: DMS-1855503 (Conant); DMS-1665035, DMS-1790212 (Pillay)}

\address{Department of Mathematics\\
University of Notre Dame\\
Notre Dame, IN 46556\\
 USA}
\email{apillay@nd.edu}

\begin{document}

\begin{abstract} We combine the fundamental results of Breuillard, Green, and Tao \cite{BGT} on the structure of approximate groups, together with  ``tame" arithmetic regularity  methods based on work of the authors and Terry \cite{CPTNIP}, to give a structure theorem for finite subsets $A$ of arbitrary groups $G$ where $A$ has ``small tripling" and bounded VC-dimension: Roughly speaking, up to a small error, $A$ will be a union of a bounded number of translates of a coset nilprogression of bounded rank and step (see Theorem \ref{thm:NIPgen}). We also prove a stronger result in the setting of bounded exponent (see Theorem \ref{thm:NIPexp}).
Our results extend recent work of Martin-Pizarro, Palac\'{i}n, and Wolf \cite{MPW} on finite stable sets of small tripling.
\end{abstract}

\maketitle

\section{Introduction}

\subsection{Small doubling in abelian groups}\label{sec:intro1} In additive combinatorics, an ``inverse theorem" is a result in which one deduces structural information for a subset $A$ of an abelian group by studying the behavior of its iterated sumsets $A+A+\ldots+A$. If $A$ is finite then, by comparing the size of $A$ to the sizes of its sumsets, one can approximate the extent to which $A$ is ``group-like" in the sense of being closed under addition. The philosophy is that if $A$ is approximately structured in this way, then it can be approximated by objects that are ``perfectly structured" such as subgroups or arithmetic progressions \cite{TaoICM}.  

A concrete example of this philosophy is the simple exercise that if $G$ is an abelian group and $A\seq G$ is a finite set satisfying $|A+A|=|A|$, then $A$ is either empty or a coset of a subgroup of $G$ (see \cite[Proposition 2.2]{TaoVu}). More generally, we say that a (nonempty) finite set $A\seq G$ has \textbf{$k$-doubling} if $|A+A|\leq k|A|$. The next result provides an approximate description of finite sets with small doubling in abelian groups.

\begin{theorem}[Green \& Ruzsa \cite{GrRuz}]\label{thm:GR}
Suppose $G$ is an abelian group and $A\seq G$ is a finite set with $k$-doubling. Then there is a proper coset progression $P\seq 2A-2A$ of rank $O_k(1)$ such that $A$ is covered by $O_k(1)$ translates of $P$.
\end{theorem}

Coset progressions, and related objects, are defined in Section \ref{sec:Gprelim2}.
The previous theorem was first proved by Freiman \cite{FreiFST} for $(\Z,+)$, and later a different proof was found by Ruzsa, who extended the result to torsion-free abelian groups and groups of bounded exponent (see \cite{Ruz94,RuzBE}).

\subsection{Small tripling in arbitrary groups}\label{sec:intro2} From a distance, Theorem \ref{thm:GR} says that if $A$ is a finite set of small doubling in an abelian group, then $A$ can be covered by a bounded number of translates of a well-structured set $P$ of bounded complexity, which lies inside a small sumset of $A$. We will refer to this kind of a result as a \emph{Freiman-Ruzsa} property. Some sources, such as \cite{SanBR}, follow the naming convention ``Bogolyubov-Ruzsa"
 in reference to the notable influence of \cite{Bog39} on Ruzsa's work. Our  terminology is chosen to be consistent with \cite[Section 2]{BGT}.
 
 In abelian groups, an important feature  of the small doubling assumption for a finite set is that it leads to controlled growth of all iterated sumsets via the \emph{Plunnecke-Ruzsa inequalities} (see Remark \ref{rem:PRI}). For nonabelian groups, this is no longer the case. Instead, one needs the stronger assumption of small \emph{tripling}. Specifically, given an arbitrary group $G$, we say that a finite set $A\seq G$ has \textbf{$k$-tripling} if $|A^3|\leq k|A|$. By results of Tao \cite{TaoPSE}, this property is closely related to the notion of a $k$-approximate group (see Definition \ref{def:approxg}). In particular, a $k$-approximate group has $k$-tripling by definition and, conversely, if $A$ has $k$-tripling then $(A\cup A\inv)^2$ is an $O(k^{O(1)})$-approximate group by \cite{TaoPSE}. Combining this with the main structure theorem for approximate groups, due to Breuillard, Green and Tao \cite{BGT}, we obtain the following analogue of Theorem \ref{thm:GR} for finite sets of small tripling in arbitrary groups. 

\begin{theorem}[Breuillard-Green-Tao \cite{BGT}, Tao \cite{TaoPSE}]\label{thm:BGT}
Suppose $G$ is a group and $A\seq G$ is a finite set with $k$-tripling. Then there is a coset nilprogression $P\seq (A\cup A\inv)^{8}$ of rank and step $O_k(1)$, and in $O_k(1)$-normal form, such that $A$ is covered by $O_k(1)$ left translates of $P$. 
\end{theorem}

In \cite{CoBogo}, the first author modified certain parts of \cite{BGT} (specifically, a combinatorial argument of Sanders \cite{SanBS}) to show that in Theorem \ref{thm:BGT}, one can in fact obtain $P\seq A^2A^{\nv 2}\cap A^{\nv 2}A^2\cap (AA\inv)^2\cap (A\inv A)^2$. With this improvement, Theorem \ref{thm:BGT} fully generalizes Theorem \ref{thm:GR}. We also refer the reader to Breuillard's surveys \cite{BreuSSA} and \cite{BreuH5P}. The first paper points out the origins of ``approximate subgroup theory" in  superstrong approximation (as well as in the inverse problems from Section \ref{sec:intro1}). 

\subsection{Bounded exponent}\label{sec:intro3} It is also common in additive and multiplicative combinatorics to combine the small doubling or tripling condition with further restrictions, such as a bound on the exponent of the ambient group. In the abelian case, this reflects the frequent use of groups of the form $\F_p^n$ as ``toy models" for problems about general abelian groups. The next result provides an analogue of Theorem \ref{thm:BGT} for groups of bounded exponent. 

\begin{theorem}\textnormal{\cite{RuzBE,HruAG,BGT,CoBogo}}\label{thm:BGTbe}
Suppose $G$ is a group of exponent $r$ and $A\seq G$ is a finite set with $k$-tripling. Then there is a subgroup $H\seq A^2A^{\nv 2}\cap A^{\nv 2}A^2\cap (AA\inv)^2\cap (A\inv A)^2$ such that $A$ is covered by $O_{k,r}(1)$ left cosets of $H$.
\end{theorem}

In the abelian case, this theorem is due to Ruzsa \cite{RuzBE}. The general case was proved by Hrushovski \cite{HruAG}, and again later by Breuillard, Green, and Tao \cite{BGT} under the stronger assumption that $A$ is a $k$-approximate group. Via the result of Tao \cite{TaoPSE} mentioned above, this implies Theorem \ref{thm:BGTbe} for $A$ with $k$-tripling, but with $H\seq (A\cup A\inv)^{8}$, and the final improvement is done in \cite{CoBogo}. Also, via an observation of van den Dries \cite{vdDag}, it suffices to only assume that every element in the set $A^2A^{\nv 2}\cap A^{\nv 2}A^2\cap (AA\inv)^2\cap (A\inv A)^2$ has order at most $r$ (see Proposition \ref{prop:GYbd}).

\section{Main results}

\subsection{NIP and VC-dimension}\label{sec:intro4} In the present article, we consider a structural assumption on subsets of groups defined using VC-dimension. To see the relevance of this assumption in the context of finding ``group-like" structure, we formulate VC-dimension for subsets of groups in terms of forbidden bipartite graphs. We use the notation $\Gamma=(V,W;E)$ for bipartite graphs, where $V$ and $W$ are vertex sets and $E\seq V\times W$ is an edge set. Given a group $G$ and a set $A\seq G$, let $\Gamma_G(A)$ denote the bipartite graph $(G,G;E_A)$ where $E_A=\{(x,y)\in G^2:yx\in A\}$. 

A simple exercise, in a similar spirit as the one preceding Theorem \ref{thm:GR}, is that if $G$ is an \emph{arbitrary} group and $A\seq G$ is a nonempty set of \emph{any cardinality}, then $A$ is a coset of a subgroup of $G$ if and only if $\Gamma_G(A)$ omits $([2],[2];\leq)$ (by which we always mean as an \emph{induced subgraph}).\footnote{We use $[n]$ to denote $\{1,\ldots,n\}$.} Thus, in addition to the notion of $k$-doubling or tripling, we can use omitted subgraphs in $\Gamma_G(A)$ as a test for finding ``group-like" structure in $A$. In order to work with a numerical parameter, we say that $A\seq G$ is \textbf{$d$-NIP}\footnote{This acronym, which comes from model theory, stands for the ``negation of the independence property".} if $\Gamma_G(A)$ omits the graph $([d],\cP([d]);\in)$. A good exercise is  that if $\Gamma_G(A)$ omits some finite bipartite graph $(V,W;E)$, then $A$ is $d$-NIP for some $d\leq |V|+\lceil \log_2|W|\rceil$. Therefore we have successfully captured the phenomenon of forbidden bipartite subgraphs in a numerical way. This measurement is also mathematically useful since it  provides access to tools from VC-theory. Indeed, it follows from the definitions that $A\seq G$ is $d$-NIP if and only if the collection of left translates of $A$ has VC-dimension strictly less than $d$ (as a set system on $G$). We will also see that NIP sets small tripling are even closer to approximate groups than what is given by \cite{TaoPSE}. Specifically, if $A\seq G$ is $d$-NIP with $k$-tripling, then $A\cup A\inv\cup\{1\}$ is an $O_d(k^{O(1)})$-approximate group (see Theorem \ref{thm:TaoNIP}). 

The following is our main result.

\begin{theorem}[main result]\label{thm:NIPgen}
Suppose $G$ is a group and $A\seq G$ is a finite $d$-NIP set with $k$-tripling. Given $\epsilon>0$, there is a coset nilprogression $P\seq G$ of rank and step $O_{d,k,\epsilon}(1)$ in $O_{d,k,\epsilon}(1)$-normal form, and a set $Z\seq AP$ with $|Z|<\epsilon|A|$, satisfying the following properties.
\begin{enumerate}[$(i)$]
\item $P\seq AA\inv\cap A\inv A$ and $A\seq CP$ for some $C\seq A$ with $|C|\leq O_{d,k,\epsilon}(1)$. 
\item There is a set $D\seq C$ such that $|(A\smd DP)\backslash Z|<\epsilon|P|$.
\item If $g\in G\backslash Z$ then $|gP\cap A|<\epsilon|P|$ or $|gP\cap A|>(1-\epsilon)|P|$.
\end{enumerate}
Moreover, if $G$ is abelian then $P$ is a proper coset progression.
\end{theorem}

Before continuing, let us compare Theorem \ref{thm:NIPgen} to the other results mentioned above. First, condition $(i)$ mirrors the statement of Theorem \ref{thm:BGT}, except that we have found the coset nilprogression $P$ inside the smaller product set $AA\inv\cap A\inv A$. Beyond this, by introducing the $\epsilon$ parameter, we obtain an even stronger structural approximation of $A$ in terms of $P$. In particular, condition $(ii)$ says that $A$ looks approximately like a union of boundedly many translates of $P$. Indeed, this condition implies the simpler expression 
\[
|A\smd DP|<\epsilon(|P|+|A|).
\]
Here one might find it desirable to have an error bound in terms of $A$ only. Thus it is worth noting that, since $P\seq AA\inv\cap A\inv A$, we can used the Plunnecke-Ruzsa inequalities (Theorem \ref{thm:TaoPSE}$(a)$) to conclude that $|P|\leq k^{O(1)}|A|$. Finally, condition $(iii)$ says that almost all translates of $P$ behave regularly with respect to $A$ in the strong sense of being almost disjoint from $A$ or almost contained in $A$.

If we impose a bounded exponent assumption then, as in Section \ref{sec:intro3}, we obtain a stronger version of Theorem \ref{thm:NIPgen} in which the coset nilprogression is replaced by a subgroup. Given a group $G$, we say that a subset $X\seq G$ has \textbf{exponent at most $r$} if every element of $X$ has order at most $r$.  

\begin{theorem}\label{thm:NIPexp}
Suppose $G$ is a group and $A\seq G$ is a finite $d$-NIP set with $k$-tripling. Assume $AA\inv\cap A\inv A$ has exponent at most $r$. Given $\epsilon>0$, there is a subgroup $H\leq G$ and a set $Z\seq AH$, which is a union of  left cosets of $H$ with $|Z|<\epsilon|A|$, satisfying the following properties.
\begin{enumerate}[$(i)$]
\item $H\seq AA\inv\cap A\inv A$ and $A\seq CH$ for some $C\seq A$ with $|C|\leq O_{d,k,r,\epsilon}(1)$. 
\item There is a set $D\seq C$ such that $|(A\backslash Z)\smd DH|<\epsilon|H|$.
\item If $g\in G\backslash Z$ then $|gH\cap A|<\epsilon|H|$ or $|gH\cap A|>(1-\epsilon)|H|$.
\end{enumerate}
Moreover, $H$ is a finite Boolean combination\footnote{The ``complexity" of this Boolean combination can also be bounded in terms of the parameters $d$, $k$, $r$, and $\epsilon$ only; see Remark \ref{rem:complex}.} of bi-translates of $A$.
\end{theorem}

We now discuss how the two theorems above fit into previous work on NIP sets in groups. The additional ``structure and regularity" properties given by conditions $(i)$ and $(iii)$ of Theorems \ref{thm:NIPgen} and \ref{thm:NIPexp} are based on related results concerning ``tame" arithmetic regularity for subsets of finite groups. The notion of arithmetic regularity was developed by Green \cite{GreenSLAG} to provide a group-theoretic analogue of Szemer\'{e}di regularity for graphs. In \cite{TeWo}, Terry and Wolf proved a strong form of arithmetic regularity for \emph{$d$-stable} subsets of $\F_p^n$ (see Section \ref{sec:stable} for details on stability). These results were then qualitatively generalized to arbitrary finite groups by the authors and Terry in \cite{CPT}. A quantitative generalization for finite abelian groups was done later by Terry and Wolf \cite{TeWo2}, and recently for arbitrary finite groups by the first author \cite{CoQSAR}.

Shortly after \cite{TeWo} and \cite{CPT}, Alon, Fox, and Zhao \cite{AFZ} proved a similar arithmetic regularity result for NIP sets in finite abelian groups of bounded exponent. This was soon followed by work of Sisask \cite{SisNIP} on NIP sets in finite abelian groups and, independently, work of the authors and Terry \cite{CPTNIP} on NIP sets in arbitrary finite groups. The work in \cite{CPTNIP} used model theoretic tools developed by the authors in \cite{CPpfNIP}, which will also play a key role in the present paper. 

 For infinite groups, a ``tame" Freiman-Ruzsa result for \emph{stable} sets of small tripling was recently obtained by Martin-Pizarro, Palac\'{i}n, and Wolf \cite{MPW} (see Theorem \ref{thm:MPW} below). We also note that Sisask's earlier work in \cite{SisNIP} contains a strong Freiman-Ruzsa component. For example, \cite[Theorem 5.4]{SisNIP} provides a result along the lines of Theorem \ref{thm:NIPexp} in the case that $G=\F_p^n$ for some $n\geq 1$ and fixed prime $p$.
 
 \subsection{Sketch of the proofs}\label{sec:sketch}
 
 For the most part, our proofs will rely on model theoretic tools involving ultraproducts and pseudofinite sets. Note that Theorems \ref{thm:NIPgen} and \ref{thm:NIPexp} are asymptotic statements about finite subsets of groups. Thus if we assume that such a statement is false, then we obtain an infinite sequence of counterexamples, which can be traded for a single limiting structure via the ultraproduct construction. So that we can more easily apply model theoretic methods, we will pass to a sufficiently saturated elementary extension of this ultraproduct in an appropriate first-order language. In the end, we will have a saturated group $G$ and a pseudofinite definable set $A\seq G$, along with the $A$-normalized pseudofinite counting measure $\mu_A$ on definable subsets of $G$. In this setting, ``small tripling" manifests as the property that $\mu_A(A^3)<\infty$. By the Plunnecke-Ruzsa inequalities, this is \emph{equivalent} to the property that any definable subset of $\langle A\rangle$ has finite $\mu_A$-measure.

The above setting is very similar to the ones used by Hrushovski \cite{HruAG} and Breuillard, Green, and Tao \cite{BGT} in their work on approximate groups. So let us assume for the moment that $A$ is a (pseudofinite) approximate group which, in this setting, means that $1\in A$, $A=A\inv$, and $A^2$ is covered by finitely many translates of $A$. Then there are two main steps in the work from \cite{HruAG} and \cite{BGT}. The first step is to show that $A^4$ contains a type-definable subgroup $\Gamma$ of $G$, which is normal and bounded index in the subgroup generated by $A$. (This is step called the ``stabilizer theorem", as it is related to the general theory of stabilizers of types in definable groups, and is in fact close to the case of groups definable in simple theories.)  Consequently, if $X\seq G$ is any definable set containing $\Gamma$, then $A$ can be covered by finitely many translates of $X$. Since $\Gamma\seq A^4$, we can choose $X\seq A^4$ and so, altogether, this yields a pseudofinite Freiman-Ruzsa property for the set $X$ with respect to $A$. It remains to find such a set $X$ exhibiting desirable algebraic properties, and this is the second main step. The key idea is to replace $X$ by a definable set admitting some kind of structural nilpotence, e.g., via coset nilprogressions in \cite{BGT} or a more technical ``Bourgain chain condition" in \cite{HruAG}. This part of the proof crucially relies on the underlying topological structure of $\langle A\rangle/\Gamma$ as a locally compact Hausdorff group, which can be ``modeled" by Lie groups via Gleason-Yamabe type results. Actually in \cite{BGT} it is shown that that the connected component of the topological group  $\langle A\rangle/\Gamma$ is residually nilpotent, which explains the role of coset nilprogressions. 

We now discuss some of the new aspects of the proofs in this paper, pointing out that they are  more than just a direct combination of \cite{BGT} and \cite{CPTNIP}.  We return to the setting in which $A$ has small tripling and  is NIP.
We construct the subgroup $\Gamma$ as above, in a different way (compared 
to \cite{HruAG}), bearing in mind that $A$ is NIP and giving local (formula-by-formula) versions of stabilizer theorems for definable groups in NIP theories. In particular, we show that $\Gamma$ is contained in the ``smaller" product set $AA\inv\cap A\inv A$.  This appears in Section 5.

Another key tool is ``generic locally compact domination" in place of the ``generic compact domination" from \cite{CPTNIP}, which means in the model-theoretic sense working with ind-definable groups in place of definable groups, to show that definable sets approximating $\Gamma$ behave regularly with respect to $A$. These results appear in Section 6 (see Corollary \ref{cor:oneY} in particular) and, together with results from \cite{BGT} on pseudofinite coset nilprogressions (see Theorem \ref{thm:BGTpf}), they are enough to prove Theorem \ref{thm:NIPgen}. In order to prove Theorem \ref{thm:NIPexp}, we use a standard Gleason-Yamabe argument, which says that if $AA\inv\cap A\inv A$ has finite exponent then $\Gamma$ can be written as an intersection of definable \emph{subgroups} of $G$.

 \subsection{The abelian case (and a note on bounds)} 
 Note that we have omitted all mention of explicit bounds in the results discussed above. The situation is as follows. For results on abelian groups proved using additive combinatorics, one obtains explicit (and often quite efficient bounds). We refer the reader to the relevant sources for further details. On the other hand, with the exception of \cite{CoQSAR}, all of the results mentioned above involving arbitrary groups rely on an ultraproduct construction in one way or another, and thus often lead to no explicit information about bounds. This partially includes Theorem \ref{thm:BGT}, in the sense that the bound on the number of translates of $P$ needed to cover $A$ is not effective. However, if one relaxes $(A\cup A\inv)^{8}$ to $(A\cup A\inv)^{24}$, then one can bound the rank and step of $P$ by $O(k^2\log k)$ (see \cite[Theorem 2.12]{BGT}). 
 
Given the above discussion, it is natural to expect that, for the case of abelian groups,  it should be possible to prove Theorems \ref{thm:NIPgen} and \ref{thm:NIPexp}  via  finitary methods. As a step in this direction, we will show in Section \ref{sec:abelian} how to prove the abelian cases of these theorems directly from the relevant regularity results for finite groups, together with a result of Green and Ruzsa \cite{GrRuz} that a finite set of small doubling in an abelian group can be Freiman-isomorphically modeled by a dense set in a finite abelian group. Combined with quantitative results of  Alon, Fox, and Zhao \cite{AFZ}, this yields a quantitative version of Freiman-Ruzsa  for NIP sets of small doubling in abelian groups of bounded exponent (i.e., the abelian case of Theorem \ref{thm:NIPexp}). In order to obtain explicit bounds for NIP sets of small doubling in arbitrary abelian groups, one would need an effective version of Theorem \ref{thm:CPTNIPab} below. We expect that such a result could be in reach via the tools developed by Sisask in \cite{SisNIP}.

\section{Combinatorial preliminaries}

\subsection{Notation and terminology}\label{sec:Gprelim1}

Given positive integers $m,n$, and some set $X$ of parameters, we use the notation $m\leq O_X(n)$ to mean that $m\leq cn$ where $c$ is a positive constant depending only $X$. If $X=\emptyset$ then $c$ is an absolute constant and we write $m\leq O(n)$.

Let $G$ be a group. We use concatenation for the group operation in $G$, and we let $1$ denote the identity in $G$.  Given sets $A,B\seq G$, we let $AB=\{ab:a\in A,~b\in B\}$. This  generalizes to products of any finite number of sets in the obvious way. Given $A\seq G$ and $g\in G$, we use $gA$ (resp., $Ag$) for $\{g\}A$ (resp., $A\{g\}$), and we call this set a \textbf{left} (resp., \textbf{right}) \textbf{translate} of $A$.  

Given $A\seq G$, we let $A^0=\{1\}$ and, by induction, $A^{n+1}=AA^n$. Similarly, we let $A^{\nv 1}=\{a\inv:a\in A\}$, and set $A^{\nv n}=(A\inv)^n$. Given $n\geq 1$ and $A\seq G$, we let $A^{\pm n}$ denote $(A\cup A\inv)^n\cup \{1\}$. Note that $\langle A\rangle=\bigcup_{n\geq 0}A^{\pm n}$. 

If $G$ is abelian then we will switch to additive notation involving $+$ for the group operation and $0$ for the identity.  

The following are several properties of subsets of groups that will be used throughout the paper.

\begin{definition}\label{def:approxg}
Let $G$ be a group and fix a nonempty set $A\seq G$.
\begin{enumerate}
\item $A$ is \textbf{symmetric} if $1\in A$ and $A=A\inv$ (equivalently, if $A=A^{\pm 1}$).
\item Given $k\geq 1$, we say $A$ has \textbf{$k$-tripling} if $A$ is finite and $|A^3|\leq k|A|$.
\item Given $k\geq 1$, we say $A$ is a \textbf{$k$-approximate group} if $A$ is symmetric and $A^2$ can be covered by $k$ left translates of $A$.
\end{enumerate}
\end{definition}

The following result combines Lemma 3.4 and Theorem 3.9 of \cite{TaoPSE}.

\begin{theorem}[Tao \cite{TaoPSE}]\label{thm:TaoPSE}
Let $G$ be a group, and suppose $A\seq G$ has $k$-tripling. 
\begin{enumerate}[$(a)$]
\item If $n\geq 1$ and $\epsilon_1,\ldots,\epsilon_n\in\{1,\nv1\}$, then $|A^{\epsilon_1}\ldots A^{\epsilon_n}|\leq k^{O_{n}(1)}|A|$.
\item $A^{\pm 2}$ is an $O(k^{O(1)})$-approximate group.
\end{enumerate}
\end{theorem}

\begin{remark}\label{rem:PRI}
The family of inequalities in Theorem \ref{thm:TaoPSE}$(a)$ are often referred to as the \emph{Plunnecke-Ruzsa inequalities} since, in the case of abelian groups, they were first proved by Plunnecke when $\epsilon_i=1$ for all $i$, and then by Ruzsa for any $\epsilon_i$. These proofs for abelian groups yield $O_n(1)=n$ and, moreover, only require $k$-doubling for the set $A$. For nonabelian groups, one may take $O_n(1)=2n$ by \cite[Theorem 5.1]{RuzIrr}. See also \cite{PetPI} for short proofs of these inequalities (with comparable bounds) by Petridis. Up to minor changes in the bounds, the same inequalities hold under the weaker assumption that $A$ has $k$-doubling and $|AaA|\leq k|A|$ for all $a\in A$. 
 
 Altogether, for our purposes, one should view the Plunnecke-Ruzsa inequalities as the primary objective, and assumptions such as small doubling or tripling as means to that end.
\end{remark}

\subsection{Coset nilprogressions}\label{sec:Gprelim2}

In this section we recall the definition from \cite{BGT} of a \emph{coset nilprogression}. We start with the simpler notion of a \emph{generalized progression}.

\begin{definition}[Generalized progression]
Let $G$ be a group. Given real numbers $L_1,\ldots,L_r>0$, let $W(L_1,\ldots,L_r)$ denote the collection of group words\footnote{A \emph{group word} is a term in the language of groups with a function symbol for inversion.}  $w(x_1,\ldots,x_r)$ in variables $x_1,\ldots,x_r$ such that, for all $1\leq i\leq r$,  the terms $x_i$ and $x_i\inv$ appear in $w$ at most $L_i$ times. Given group elements $u_1,\ldots,u_r\in G$, define
\[
P(u_1,\ldots,u_r;L_1,\ldots,L_r)=\{w(u_1,\ldots,u_r):w\in W(L_1,\ldots,L_r)\}.
\]
A \textbf{generalized progression} in $G$ is a subset $P$ of the form $P(u_1,\ldots,u_r,L_1,\ldots,L_r)$ for some choice of $\bar{u}$ and $\bar{L}$. In this case, we call $r$ the \textbf{rank} of $P$. By convention, $\{1\}$ is a generalized progression in $G$ of rank $0$. 
\end{definition}

The decision to allow $L_1,\ldots,L_r$ to be real numbers (rather than just integers) will be relevant in  Proposition \ref{prop:scaleCN} below.
In \cite{BGT}, generalized progressions are called \emph{non-commutative progressions} to emphasize the setting of arbitrary groups. Indeed, for abelian groups the above definition simplifies to the more familiar notion of a \emph{generalized arithmetic progression}.

\begin{definition}[Generalized arithmetic progression]
Let $G$ be an abelian group. A \textbf{generalized arithmetic progression} in $G$ is a subset $P$ of the form
\[
P(u_1,\ldots,u_r;L_1,\ldots,L_r)=\{n_1u_1+\ldots+n_ru_r:|n_i|\leq L_i\text{ for all }1\leq i\leq r\}.
\]
In this case, we call $r$ the \textbf{rank} of $P$. 
\end{definition} 

This definition of a generalized arithmetic progression differs slightly from other sources (such as \cite{GrRuz} and \cite{TaoVu}) where, for example, such progressions are allowed to have an affine term (see also Remark \ref{rem:progsym}). 

Freiman's Theorem in the integers \cite{FreiFST, Ruz94} says that finite sets in $\Z$ with $k$-doubling can be approximated by generalized arithmetic progressions of rank $O_k(1)$. This  fails for arbitrary abelian groups, where one may have nontrivial finite subgroups. In \cite{GrRuz}, Green and Ruzsa formulate the notion of a \emph{coset progression} to address this issue.

\begin{definition}[Coset progression]
Let $G$ be an abelian group. A \textbf{coset progression} in $G$ is a set $P$ of the form $P_0+H$ where $P_0\seq G$ is a generalized arithmetic progression and $H$ is a finite subgroup of $G$. In this case, we say that $P$ has \textbf{rank} $r$ if $P_0$ has rank $r$.
\end{definition}

A finite subgroup of an abelian group is a coset progression of rank $0$. Moreover, a coset progression of rank $r$ (in an abelian group) is a $(2^r+1)$-approximate group. 
On the other hand, it is not necessarily the case that a generalized progression in a nonabelian group is an approximate subgroup (see the discussion after \cite[Remark 2.4]{BGT} for precise details). Indeed, a main aspect of the results from \cite{BGT} and \cite{HruAG} is that approximate groups exhibit a certain kind of nilpotent structure, which is made precise in \cite{BGT} as follows. 

\begin{definition}[Nilprogression]
Let $G$ be a group. Given elements $u_1,\ldots,u_r\in G$, set $c_1(\ubar)=\{u_1,\ldots,u_r,u_1\inv,\ldots,u\inv_r\}$ and, for $n>1$, inductively define $c_n(\ubar)=\{[g,h]:g\in c_j(\ubar),~h\in c_k(\ubar),~j+k=n\}$. 

 A \textbf{nilprogression} in $G$  is a generalized progression $P=P(u_1,\ldots,u_r;L_1,\ldots,L_r)$ such that $c_{s+1}(\ubar)=\{1\}$ for some $s\geq 0$. In this case, we say that $P$ has \textbf{rank} $r$ and \textbf{step} $s$. 
\end{definition}

For example, a generalized arithmetic progression in an abelian group is a nilprogression of step $1$. We can now state the appropriate analogue of coset progressions for arbitrary groups.

\begin{definition}[Coset nilprogression]\label{def:nilp} Let $G$ be a group. A \textbf{coset nilprogression} in $G$ is a set $P$ of the form $P_0H$ where $P_0\seq G$ is a nilprogression and $H$ is a finite subgroup of $G$ normalized by $P_0$.  In this case, we say that $P$ has \textbf{rank} $r$ and \textbf{step} $s$ if $P_0$ has rank $r$ and step $s$.
\end{definition}

Note that a finite subgroup of an arbitrary group is a coset nilprogression of rank and step $0$.

\begin{remark}\label{rem:progsym}
Any coset nilprogression in a group $G$ is symmetric.
\end{remark}

The rest of this section deals with tools for controlling the sizes and multiplicative behavior of nilprogressions. As before, we start with the abelian case.

\begin{definition}[properness]
Let $G$ be an abelian group. A generalized arithmetic progression $P(u_1,\ldots,u_r;L_1,\ldots,L_r)\seq G$ is \textbf{proper} if the elements $n_1u_1+\ldots+n_ru_r$ are pairwise distinct for varying choices of $n_1,\ldots,n_r$. A coset progression $P=P_0H\seq G$ is \textbf{proper} if $P_0$ is proper.
\end{definition}

Note that a generalized arithmetic progression $P=P(u_1,\ldots,u_r;L_1,\ldots,L_r)$ in an abelian group is proper if and only if $|P|=\prod_{i=1}^r(2\lfloor L_i\rfloor+1)$. In arbitrary groups, this behavior is managed in an asymptotic fashion via the notion of ``normal form". Before stating this definition, we need to slightly shift the perspective on coset nilprogressions.

\begin{remark}\label{rem:pushCN}
Suppose $G$ is a group and $P=P_0H$ is a coset nilprogression in $G$, where $P_0$ is a nilprogression  of rank $r$ and step $s$ and $H$ is normalized by $P_0$. Then $H$ is a normal subgroup of $\langle P\rangle$ and, if $\pi\colon\langle P\rangle\to \langle P\rangle/H$ is the quotient map, then $\pi(P)$ is a nilprogression in $\langle P\rangle/H$ of rank $r$ and step $s$. In particular, if $P_0=P(u_1,\ldots,u_r;L_1,\ldots,L_r)$ then $\pi(P)=P(u_1H,\ldots,u_rH;L_1,\ldots,L_r)$. 
\end{remark}

\begin{definition}[$c$-normal form]\label{def:cnf}
Let $G$ be a group and fix an integer $c\geq 1$. A generalized progression $P=P(u_1,\ldots,u_r;L_1,\ldots,L_r)$ in $G$ is in \textbf{$c$-normal form} if:
\begin{enumerate}[$(i)$]
\item For any $1\leq i<j\leq r$,
\[
\textstyle[u_i,u_j],[u_i\inv,u_j],[u_i\inv,u_j],[u_i\inv,u_j\inv]\in P\left(u_{j+1},\ldots,u_r;\frac{cL_{j+1}}{L_iL_j},\ldots,\frac{cL_r}{L_iL_j}\right).
\]
\item The elements $u_1^{n_1}\ldots u_r^{n_r}$ are distinct for distinct choices of $n_i\leq c\inv |L_i|$.
\item $c\inv\prod_{i=1}^r (2\lfloor L_i\rfloor+1)\leq |P|\leq c\prod_{i=1}^r(2\lfloor L_i\rfloor+1)$.
\end{enumerate}
 A coset nilprogression $P$ is in \textbf{$c$-normal form} if the corresponding nilprogression $\pi(P)$ from Remark \ref{rem:pushCN} is in $c$-normal form.
\end{definition}

 We will not need to delve into Definition \ref{def:cnf} in any detail, except to note that a coset progression in an abelian group is proper if and only if it is in $1$-normal form. Our main use of the normal form condition will be to control the sizes of ``scaled" coset nilprogressions via Proposition \ref{prop:scaleCN} below.

\begin{definition}
Let $G$ be a group, and suppose $P=P_0H$ is a coset nilprogression, where $P_0=P(u_1,\ldots,u_r;L_1,\ldots,L_r)$. Given a real number $\epsilon>0$, let $P^{(\epsilon)}$ denote the coset nilprogression $P(u_1,\ldots,u_r;\epsilon L_1,\ldots,\epsilon L_r)H$. 
\end{definition}

The next result proposition follows from  \cite[Lemma C.1]{BGT}. 

\begin{proposition}\label{prop:scaleCN}
Suppose $G$ is a group and $P\seq G$ is a coset nilprogression of rank $r$ in $c$-normal form. Then, for any $\epsilon>0$, one has $|P|\leq O_{r,c,\epsilon}(|P^{(\epsilon)}|)$. 
\end{proposition}

\subsection{Set systems and VC-dimension}

 \begin{definition}
 Let $X$ be a set and fix $\cS\seq\cP(X)$. Then $\cS$ \textbf{shatters} a subset $A\seq X$ if $\cP(A)=\{A\cap S:S\in\cS\}$. The \textbf{VC-dimension} of $\cS$, denoted $\VC(\cS)$, is the supremum over all $n\in\N$ such that $\cS$ shatters a subset of $X$ of size $n$. Note that $\VC(\cS)$ takes a value in $\N\cup\{\infty\}$.
 \end{definition}
 
 The following is a corollary of the  ``VC Theorem for finite set systems", proved by Vapnik and Chervonenkis in \cite{VCdim} (see also \cite[Corollary 6.9]{Sibook}). 
 
 \begin{theorem}\label{thm:VC}\textnormal{\cite{VCdim}}
 Let $(X,\mu)$ be a finite probability space, and suppose $\cS$ is a collection of measurable subsets of $X$ with $\VC(\cS)\leq d<\infty$. Then, for any $0<\epsilon<1$, there is a sequence $(x_1,\ldots,x_n)\in X^n$, with $n\leq O_{d,\epsilon}(1)$, such that for any $S\in\cS$,
 \[
 \textstyle\left|\mu(S)-\frac{1}{n}|\{1\leq i\leq n:x_i\in S\}|\right|\leq\epsilon.
 \]
 \end{theorem}
 
Next we state Matou\v{s}ek's ``$(p,q)$-theorem" (see \cite[Theorem 4]{MatFHP}). This result will not be necessary for the proof of our main results, and will only be used in Section \ref{sec:unique}. 
 
 \begin{theorem}\label{thm:Mat}\textnormal{\cite{MatFHP}}
Let $X$ be a set and suppose $\cS\seq\cP(X)$ is finite with $\VC(\cS)\leq d<\infty$. Fix $p\geq q\geq 2^{d+1}$, and suppose that among any $p$ sets in $\cS$ there are $q$ sets with nontrivial intersection. Then there is some $F\seq X$ such that $|F|\leq O_{p,q}(1)$ and $F\cap S\neq\emptyset$ for all $S\in\cS$.
 \end{theorem}

 Finally, we list some standard facts about various operations on set systems.
 
 \begin{fact}\label{fact:VCdim}
 Let $X$ be a set.
 \begin{enumerate}[$(a)$]
 \item If $\cS\seq \cP(X)$ and $\cS'=\{X\backslash S:S\in\cS\}$, then $\VC(\cS)=\VC(\cS')$.
 \item If $\cS_1,\cS_2\seq\cP(X)$ are such that $\VC(\cS_i)\leq d<\infty$, and $\cS=\{S_1\cap S_2:S_i\in\cS_i\}$, then $\VC(\cS)<10d$.
 \item If $\cS\seq\cP(X)$ is such that $\VC(\cS)<\infty$ and $\cS^*=\{\cS_x:x\in X\}$ where, given $x\in X$, $\cS_x=\{S\in\cS:x\in S\}$,  then $\VC(\cS^*)<2^{\VC(\cS)+1}$.
 \end{enumerate}
 \end{fact} 
 \begin{proof}
 Part $(a)$ is easy and part $(c)$ is a straightforward exercise (see, e.g., \cite[Lemma 6.3]{Sibook}). Part $(b)$ is evident from \cite[Theorem 9.2.6]{DudNIP}, but we will sketch the proof to clarify the numerics. Given any $\cS\seq\cP(X)$, define the \emph{shatter function} $\pi_{\cS}\colon\N\to\N$ such that $\pi_{\cS}(n)=\max\{|\{A\cap S:S\in\cS\}|:A\seq X,~|A|=n\}$. 
By the Sauer-Shelah Lemma (see, e.g.,  \cite[Lemma 6.4]{Sibook}), if $\VC(\cS)\leq d$ then $\pi_{\cS}(n)\leq(en/d)^d$ for any $n\geq d$.  Now let $\cS_1$, $\cS_2$, and $\cS$ be as above. One can check that $\pi_{\cS}(n)\leq \pi_{\cS_1}(n)\pi_{\cS_2}(n)$ for all $n\geq 0$, and so $\pi_{\cS}(n)\leq (en/d)^{2d}$ for all $n\geq d$. In particular, if $n\geq d$ and $(en/d)^{2d}<2^n$, then $\pi_{\cS}(n)<2^n$, which implies $\VC(\cS)<n$. Setting $n=10d$ satisfies these constraints.
 \end{proof}

Recall from Section \ref{sec:intro4} that a subset $A$ of a group $G$ is $d$-NIP if and only if the collection of left translates of $A$ has VC-dimension strictly less than $d$. We say that $A\seq G$ is \textbf{NIP} if it is $d$-NIP for some $d\geq 1$. The previous fact has the following consequences for NIP subsets of groups.

\begin{corollary}\label{cor:VCG}
Let $G$ be a group.
\begin{enumerate}[$(a)$]
\item The collection of NIP subsets of $G$ is a bi-invariant Boolean algebra.
\item If $A\seq G$ is $d$-NIP then the collection of right translates of $A$ has VC-dimension at most $2^d$.
\end{enumerate}
\end{corollary}
\begin{proof}
Part $(a)$ is immediate from Fact \ref{fact:VCdim}$(a,b)$ and the easy exercise that a translate of an NIP subset of $G$ is NIP. For part $(b)$, fix $A\seq G$ and define $\cS_\ell=\{gA:g\in G\}$ and $\cS_r=\{Ag:g\in G\}$. Then one can check that $\VC(\cS_r)=\VC(\cS_\ell^*)$, and so the result follows from Fact \ref{fact:VCdim}$(c)$.
\end{proof}

 \subsection{NIP sets of small tripling are essentially approximate groups}
 
 In this section, we prove a useful strengthening of Theorem \ref{thm:TaoPSE}$(b)$ in the NIP setting. 

  \begin{theorem}\label{thm:TaoNIP}
 Let $G$ be a group, and suppose $A\seq G$ is $d$-NIP with $k$-tripling. Then $A^{\pm 2}\seq EA\cap AF$ for some $E,F\seq A^{\pm 3}$ of size $O_d(k^{O(1)})$. In particular, $A^{\pm 1}$ is an $O_d(k^{O(1)})$-approximate group.
 \end{theorem}
 \begin{proof}
 Set $X=A^{\pm 3}$, and note that $X$ is finite. Let $\cS=\{Ag:g\in A^{\pm 2}\}$, and note that $\cS$ is a collection of subsets of $X$.  By Theorem \ref{thm:TaoPSE}$(a)$, there is $\epsilon=k^{\nv O(1)}$ such that, for any $S\in\cS$, we have $|S|=|A|\geq \epsilon|X|$. By Corollary \ref{cor:VCG}$(b)$, we have $\VC(\cS)\leq 2^d$.  By Theorem \ref{thm:VC}, there is a set $E\seq X$ such that $|E|\leq O(2^d\epsilon^{\nv 2} \log(2^d\epsilon\inv))\leq O_d(k^{O(1)})$ and $S\cap E\neq\emptyset$ for all $S\in\cS$.
 Therefore, if $g\in A^{\pm 2}$ then $Ag\inv\cap E\neq\emptyset$, i.e., $g\in x\inv A$ for some $x\in E$.  So $A^{\pm 2}\seq E\inv A$. Applying the same argument to $\cS'=\{gA:g\in A^{\pm 2}\}$, which has VC-dimension at most $d-1$, we obtain $F\seq X$ such that $|F|\leq O_d(k^{O(1)})$ and $A^{\pm 2}\seq AF\inv$. 
 \end{proof}
 
 \begin{remark}
In the previous proof, we only needed to bound the VC-dimension of translates of $A$ by elements in $A^{\pm 2}$. Similar situations will arise frequently below, and so it is worth pointing out now that, for our purposes, the difference in the VC-dimensions of such set systems is negligible. Specifically, given a group $G$ and $A\seq G$, let $\cS_1=\{gA:g\in G\}$ and $\cS_2=\{gA:g\in AA\inv\}$. Then clearly $\VC(\cS_2)\leq\VC(\cS_1)$. Conversely, we have $\VC(\cS_1)\leq\VC(\cS_2)+1$ (this is observed by Sisask in \cite{SisNIP}). Indeed, suppose $\cS_1$ shatters a set $X\seq G$. After translating, we may assume $X\seq A$. So if $Y\seq X$ is nonempty, and $g\in G$ is such that $gA\cap X=Y$, then we necessarily have $g\in AA\inv$. Therefore $\cS_2$ shatters any proper subset of $X$.
 \end{remark}

\section{Model theoretic preliminaries}

We assume familiarity with basic notions from first-order model theory, including structures, types, saturation, and the ultraproduct construction. 

\subsection{Formulas and NIP}

Let $M$ be a first-order structure.

\begin{definition}
Let $\phi(\xbar;\ybar)$ be a formula (possibly with parameters from $M$).
\begin{enumerate}
\item Given $\bbar\in M^{\ybar}$, let $\phi(M;\bbar)=\{\abar\in M^{\xbar}:M\models\phi(\abar;\bbar)\}$.
\item A \textbf{instance} of $\phi(\xbar;\ybar)$ is a formula of the form $\phi(\xbar,\bbar)$ for some $\bbar\in M^{\ybar}$.
\item A \textbf{$\phi$-formula} is a (finite) Boolean combination of instances of $\phi$.
\item A subset of $M^{\xbar}$ is \textbf{$\phi$-definable} if it is defined by a $\phi$-formula.
\item A \textbf{$\langle \phi\rangle$-formula} is a Boolean combination of formulas of the form $\phi(\xbar,\ybar_i)$ where the variables $\ybar_i$ can vary. 
\item Let $\phi^*(\ybar;\xbar)$ denote the same formula $\phi(\xbar;\ybar)$, but with the roles of object and parameter variables exchanged.
\end{enumerate}
\end{definition}

Given $B\seq M$, we say that a formula $\theta(\xbar)$ is \textbf{over $B$} if any parameters appearing in $\theta(\xbar)$ lie in $B$. A set $X\seq M^{\xbar}$ is \textbf{definable over $B$} if $X=\theta(M)$ for some formula $\theta(\xbar)$ over $B$. If, moreover, $\theta(\xbar)$ is a $\phi$-formula for some formula $\phi(\xbar;\ybar)$ over $\emptyset$, then we say that $X$ is \textbf{$\phi$-definable over $B$}.

\begin{definition}
A formula $\phi(\xbar;\ybar)$ is \textbf{$d$-NIP} if there do not exist $(\abar_i)_{i\in[d]}$ in $M^{\xbar}$ and $(\bbar_\sigma)_{\sigma\seq[d]}$ in $M^{\ybar}$ such that $M\models\phi(\abar_i;\bbar_\sigma)$ if and only if $i\in \sigma$. We say that $\phi(\xbar;\ybar)$ is \textbf{NIP} if it is $d$-NIP for some $d\geq 1$.
\end{definition}

\begin{remark}
A formula $\phi(\xbar;\ybar)$ is $d$-NIP if and only if  $\{\phi(M;\bbar):\bbar\in M^{\ybar}\}$ has VC-dimension strictly less than $d$ (as a set system on $M^{\xbar}$). Note also that a subset $A$ of a group $G$ is $d$-NIP if and only if the ``formula" $A(y\cdot x)$ is $d$-NIP. 
\end{remark}

As in Corollary \ref{cor:VCG}, one can use Fact \ref{fact:VCdim} to show that NIP formulas are closed under Boolean combinations, and that $\phi(\xbar;\ybar)$ is NIP if and only if $\phi^*(\ybar;\xbar)$ is NIP.

\subsection{Locally compact type spaces}\label{sec:LCLT}

\begin{definition}
Given a set $X$, a \textbf{ring of subsets of $X$} is a nonempty collection $\cR\seq\cP(X)$ that is closed under finite unions, intersections, and relative complements (i.e., if $A,B\in\cR$ then $A\backslash B\in\cR$). 
\end{definition}

Before moving to the model theoretic setting, we set some notation for a few frequently used rings of subsets of groups.

\begin{notation}
Let $G$ be a group and suppose $A\seq G$. We let $\cR_{\!A}$, $\cR^\ell_{\!A}$, and $\cR^r_{\!A}$ denote the rings of subsets of $\langle A\rangle$ generated by $\{gAh:g,h\in\langle A\rangle\}$, $\{gA:g\in\langle A\rangle\}$, and $\{Ag:g\in\langle A\rangle\}$, respectively. 
\end{notation}

Now let $M^*$ be a  $\kappa$-saturated first-order structure for some sufficiently large cardinal $\kappa$. We say that a set $I$ is \emph{bounded} if $|I|<\kappa$. Given variables $\xbar$, let $\Def_{\xbar}(M^*)$ denote the Boolean algebra of definable subsets of $(M^*)^{\xbar}$.  Recall that a subset of $(M^*)^{\xbar}$ is said to be {\em type-definable} if it is an intersection of a bounded collection of definable sets.  And a subset of $(M^*)^{\xbar}$ is said to be {\em ind-definable}  (also called $\vee$-definable) if it is a union of a bounded number of definable sets.  Note that if $X$ is ind-definable as 
$\bigcup_{i\in I}X_{i}$, and $Y$ is a definable subset of $(M^*)^{\xbar}$ which is a subset of $X$, then $Y\subseteq \bigcup_{i\in I_{0}}X_{i}$ for some finite subset $I_{0}$ of $I$ (by saturation of $M^*$).  

We now give some refinements in the light of the above notion of a ring of subsets. 

\begin{definition}\label{def:ind}
Fix $\xbar$ and let $\cR$ be a subring of $\Def_{\xbar}(M^*)$. Then a set $X\seq M^{\xbar}$ is \textbf{$\cR$-type-definable} (resp.,  \textbf{countably $\cR$-type-definable}) if it is an intersection of boundedly (resp., countably) many sets in $\cR$.

A set $X\seq M^{\xbar}$ is \textbf{$\cR$-ind-definable} (resp., \textbf{countably $\cR$-ind-definable}) if it is a union of boundedly (resp., countably) many sets in $\cR$.
\end{definition}

In the previous definition, if $\cR$ is the Boolean algebra of $\phi$-definable sets, for some formula $\phi(\xbar;\ybar)$, then we replace $\cR$ with $\phi$ in the above terminology. If $\cR=\Def_{\xbar}(M^*)$, then we omit it. In these two cases, we will also want to be more precise with parameters.

\begin{definition}
Fix $\xbar$ and $B\seq M^*$. Then $X\seq (M^*)^{\xbar}$ is \textbf{type-definable} (resp., \textbf{countably type-definable}) \textbf{over $B$} if it is an intersection of boundedly (resp., countably) many sets that are definable over $B$. These notions are relativized to a formula $\phi(\xbar;\ybar)$ over $\emptyset$ in the obvious way, and we also have the natural analogues with respect to ind-definability.
\end{definition}

\begin{definition}
Fix $\xbar$ and let $\cR$ be a subring of $\Def_{\xbar}(M^*)$. A complete \textbf{$\cR$-type} is a subset $p\seq\cR$ such that:
\begin{enumerate}[$(i)$]
\item $\emptyset\not\in p$,
\item  $A\in p$ for some $A$, 
\item if $A,B\in p$ then $A\cap B\in p$,
\item if $A\in p$ and $B\in\cR$, with $A\seq B$, then $B\in p$, and
\item for any $A,B\in\cR$, if $A\in p$ and $B\not\in p$ then $A\backslash B\in p$.
\end{enumerate}
We let $S(\cR)$ denote the set of complete $\cR$-types.  (We may omit the adjective ``complete".)
\end{definition}

\begin{remark}
$S(\cR)$ is a totally disconnected locally compact Hausdorff space, with a basis of compact-open sets given by $[X]=\{p\in S(\cR):X\in p\}$ for all $X\in\cR$. Note that every $p\in S(\cR)$ is a partial type in the sense of the ambient structure $M^*$. On the other hand once one picks a definable set $X\in p$, then $p$ is determined by its restriction to the Boolean algebra $\cR^X=\{X\cap A:A\in\cR\}$ of subsets of $X$.  In any case both complete $\cR$-types and $\cR$-type definable sets will ``concentrate" on some definable set $X\in{\cR}$. 

The previous remarks also connect the topology on $S(\cR)$ to more familiar spaces of types over Boolean algebras. Specifically, given $X\in\cR$, the type space $S(\cR^X)$ is a profinite space under the usual Stone topology. Given $X,Y\in\cR$, if $\emptyset\neq X\seq Y$ then for any $p\in S(\cR^X)$ there is a unique $q\in S(\cR^Y)$ such that $p=\{X\cap A:A\in q\}$, and this induces a continuous injective map from $S(\cR^X)$ to $S(\cR^Y)$. These maps yield a direct system  whose direct limit is naturally identified with $S(\cR)$. The final topology on $S(\cR)$ induced by this direct system coincides with the topology described above. 
\end{remark}

Given a type $p$ and a type-definable set $X$, we write $p\models X$ to denote that $p$ implies (or concentrates on) $X$, i.e., any realization of $p$ (perhaps in a larger model) also realizes $X$ when viewed as a partial type. If, moreover, $p\in S(\cR)$ for some ring $\cR$, and $X$ is $\cR$-type-definable, then this means that $p$ actually contains a small collection of $\cR$-definable sets whose intersection is $X$. 

\begin{definition}
Let $X$ be a set and let $\cR$ be a ring of subsets of $X$. An \textbf{extended $\cR$-Keisler measure} is a nontrivial finitely additive $\R_{\geq 0}\cup\{\infty\}$-valued measure on $\cR$. An \textbf{$\cR$-Keisler measure} is an $\R_{\geq 0}$-valued extended $\cR$-Keisler measure.
\end{definition}

  Now let $\cR$ be a subring of $\Def_{\xbar}(M^*)$ and let $\mu$ be an extended $\cR$-Keisler measure. (In the situation most interesting for us, the measure will be $\R_{\geq 0}$-valued.)

\begin{definition}
$~$
\begin{enumerate}
\item An $\cR$-type $p$ is \textbf{$\mu$-wide} if $\mu(X)>0$ for any $X\in p$.
\item An $\cR$-type-definable set $X\seq M$ is \textbf{$\mu$-wide} if $\mu(D)>0$ for any $\cR$-definable set $D\seq (M^*)^{\xbar}$ containing $X$.
\end{enumerate}
\end{definition}

\begin{fact}\label{fact:wide}
Suppose $X\seq (M^*)^{\xbar}$ is $\cR$-type-definable and $\mu$-wide. Then there is a $\mu$-wide type $p\in S(\cR)$ such that $p\models X$.
\end{fact}
\begin{proof}
Let $\cF$ be the collection of $\cR$-definable subsets of $(M^*)^{\xbar}$ containing $X$. Fix some $X\in\cF$, and let $\cF^X=\{X\cap Y:Y\in\cF\}\seq\cF$. By the usual standard exercise, there is $q\in S(\cR^X)$ such that $\cF^X\seq q$ and $\mu(Y)>0$ for all $Y\in q$. Let $p\in S(\cR)$ be the unique lift of $q$. Then one easily checks that $p$ is $\mu$-wide and $\cF\seq p$.
\end{proof}

\subsection{Locally compact quotients}\label{sec:LCLT}

Let $G^*$ be a sufficiently saturated first-order expansion of a group. We work in $G^*$ (namely take $M^*$ from the previous section to be $G^*$. Suppose $\Sigma\leq G^*$ is an ind-definable subgroup of $G^*$, and $\Gamma\leq\Sigma$ is a type-definable bounded index subgroup of $\Sigma$.  Let $\pi\colon \Sigma\to\Sigma/\Gamma$ be the canonical quotient map. Then $\Sigma/\Gamma$ is a locally compact Hausdorff space under the \textbf{logic topology}, in which a subset $C\seq\Sigma/\Gamma$ is \textbf{closed} if and and only if $\pi\inv(C)\cap X$ is type-definable for any definable set $X\seq \Sigma$. The verification of this claim is a standard exercise (see \cite[Section 7]{HPP} and \cite[Section 4]{vdDag} for details in a slightly restricted setting). When $\Sigma$ is actually \emph{definable}, then $\Sigma/\Gamma$ is compact, and this situation is covered in detail in \cite{LaPi} and \cite{PilCLG}, for example. The following are some other useful exercises.

\begin{fact}\label{fact:LCLT}
Suppose $\Sigma\leq G^*$ is ind-definable, and $\Gamma\leq\Sigma$ is type-definable of bounded index.
\begin{enumerate}[$(a)$]
\item A set $K\seq \Sigma/\Gamma$ is compact if and only if $\pi\inv(K)$ is type-definable.
\item A set $K\seq \Sigma/\Gamma$ is compact-open if and only if $\pi\inv(K)$ is definable.
\item If $X\seq\Sigma$ is type-definable then $\pi(X)$ is compact.
\item If $X\seq \Sigma$ is definable then $U=\{a\Gamma\in \Sigma/\Gamma:a\Gamma\seq X\}$ is open, and $\pi\inv(U)\seq X$.
\item If $\Sigma$ is countably ind-definable and $\Gamma$ is countably type-definable, then $\Sigma/\Gamma$ is second countable.
\item If $\Gamma$ is a normal subgroup of $\Sigma$, then $\Sigma/\Gamma$ is a topological group.
\end{enumerate}
\end{fact}

We can recast the logic topology on $\Sigma/\Gamma$ in terms of the topology on types over rings of definable sets. First, we set some notation.

\begin{definition} Suppose $\cR$ is subring of $\Def(G^*)$, $\Sigma$ is an $\cR$-ind-definable subgroup of $G^*$, $\Gamma$ is an $\cR$-type-definable bounded-index subgroup of $\Sigma$, and $\cR$ is left-$\Sigma$-invariant.  Then every $p\in S(\cR)$ determines (i.e., implies) a unique left coset $a\Gamma$ of $\Gamma$ in $\Sigma$. We let  $\tau(p)$ denote this left coset $a\Gamma$, yielding a function $\tau: S(\cR) \to \sigma/\Gamma$.
\end{definition}

The existence of $\tau$ above is immediate due to our convention that $\Gamma$ (and every left coset $a\Gamma$) is $\cR$-type definable (i.e., by a small collection of formulas compared to the saturation of $G^*$) and that $p$ is a complete $\cR$-type  (over $G^*$).  Namely, our saturation condition on $G^*$, and the bounded index assumption,  implies that every coset of $\Gamma$ in $\Sigma$ (in a bigger model) has a representative in $G^*$, whereby if $a$ realizes $p$ in a bigger model, then $a$ is in the same coset of $\Gamma$ as some $b$ in $G^*$ and this information is part of the type $p$.

\begin{proposition}\label{prop:LCLT}
Let $\cR$ be a subring of $\Def(G^*)$. Suppose $\Sigma$ is an
$\cR$-ind-definable subgroup of $G^*$, $\Gamma$ is an $\cR$-type-definable bounded-index subgroup of $\Sigma$, and $\cR$ is left-$\Sigma$-invariant. Then the logic topology on $\Sigma/\Gamma$ is the finest topology for which $\tau\colon S(\cR)\to \Sigma/\Gamma$ is continuous.
\end{proposition}
\begin{proof}
We will only prove that $\tau$ is continuous (which is all we need in subsequent results), and leave the rest as an exercise.

Suppose that  $C$ is closed in the logic topology, and fix $p\not\in\tau\inv(C)$. Then $\tau(p)\cap\pi\inv(C)=\emptyset$ and so, by saturation and since $\tau(p)$ is $\cR$-type-definable, there is some $X\in\cR$ such that $\tau(p)\seq X$ and $X\cap\pi\inv(C)=\emptyset$. So $[X]$ is an open set in $S(\cR)$, which contains $p$ and is disjoint from $\tau\inv(C)$.
\end{proof}

Our main results will require us to ``localize" the topology on $\Sigma/\Gamma$ to particular formulas. In particular, we will need the following result.

\begin{theorem}\label{thm:LCLT}
Let $\cR$ be a subring of $\Def(G^*)$. Suppose $\Sigma$ is an
$\cR$-ind-definable subgroup of $G^*$, $\Gamma$ is an $\cR$-type-definable bounded-index subgroup of $\Sigma$, and $\cR$ is left-$\Sigma$-invariant. Then a set $C\seq\Sigma/\Gamma$ is closed if and only if $X\cap\pi\inv(C)$ is $\cR$-type-definable for any $\cR$-definable set $X\seq\Sigma$.
\end{theorem}
\begin{proof}
Let $\pi\colon \Sigma\to\Sigma/\Gamma$ be the quotient map. The following claim can be proved via a straightforward generalization of \cite[Lemma 4.1]{CPpfNIP}.
\medskip

\noindent\textit{Claim:} If $X\seq\Sigma$ is type-definable then $\pi\inv(\pi(X))$ is $\cR$-type-definable. 
 \medskip
 
 Now fix a definable set $X\seq\Sigma$, and set $K=\pi(X)\seq\Sigma/\Gamma$. Then $K$ is compact by Fact \ref{fact:LCLT}. Define the \emph{$\cR$-topology} on $K$ by declaring $C\seq K$ to be \emph{$\cR$-closed} if $\pi\inv(C)$ is $\cR$-type-definable. We claim that this defines a compact Hausdorff topology on $K$. The verification of this nearly identical to the case of the logic topology (e.g., following \cite{LaPi}), and crucially relies on the claim above. Note that any $\cR$-closed subset of $K$ is closed, and so the induced logic topology on $K$ refines the $\cR$-topology. Since both topologies are compact and Hausdorff, they must coincide.  
 
Now we prove the theorem. The left-to-right implication follows immediately from the fact that $\Sigma$ is $\cR$-ind-definable.

For the other direction, suppose $C\seq\Sigma/\Gamma$ is closed and fix an $\cR$-definable set $X\seq\Sigma$. Then $C'=C\cap\pi(X)$ is a closed subset of $\pi(X)$, and hence is $\cR$-closed by the above. So $\pi\inv(C')$ is $\cR$-type-definable. Since $X\cap\pi\inv(C)=X\cap\pi\inv(C')$, it follows that $X\cap\pi\inv(C)$ is $\cR$-type-definable.
\end{proof}

\begin{remark}\label{rem:LCLT}
Given the previous theorem, it is easy to verify the analogous localization of Fact \ref{fact:LCLT} to a specific ring $\cR$.
\end{remark}

\begin{definition}
Suppose $\Sigma\leq G^*$ is ind-definable. A definable set $X\seq \Sigma$ is \textbf{left} (resp., \textbf{right}) \textbf{generic in $\Sigma$} if any definable subset of $\Sigma$ is covered by finitely many left (resp., right) translates of $X$.
\end{definition}

Our use of the word \emph{generic} follows Definition 7.2 in \cite{HPP}. In some sources, a subset $X$ of an abstract group $G$ is called (left) generic if $G$ can be covered by finitely  many (left) translates of $X$, and so the reader should make a note of this distinction.

\begin{remark}\label{rem:gen-group}
Suppose $\Sigma\leq G^*$ is ind-definable, and $\Gamma\leq \Sigma$ is type-definable of bounded index. Then, by saturation and compactness, any definable subset of $\Sigma$ containing $\Gamma$ is left and right generic in $\Sigma$.
\end{remark}

\begin{remark}
In the subsequent results, we will consider ind-definable subgroups of the form $\Sigma=\langle A\rangle$ for some definable set $A\seq G^*$. In this case, a definable set $X\seq \Sigma$ is left (resp., right) generic in $\Sigma$ if and only if $A^{\pm m}$ is covered by finitely many left (resp., right) translates of $X$ for all $m\geq 1$.
\end{remark}

\subsection{Pseudofinite sets and NIP formulas}\label{sec:PFNIP}
We now review some terminology and facts about pseudofinite sets and pseudofinite counting measures. For simplicity, we work in the precise setting that will be used later to prove our main results. Specifically, let $(M_i)_{i\in I}$ be a collection of first-order structures in a common language, where $I$ is some index set. Let $\cU$ be an ultrafilter over $I$ and set $M=\prod_{\cU}M_i$. A set $X\seq M$ is \textbf{internal} if $X=\prod_{\cU}X_i$ for some sequence $(X_i)_{i\in I}$ with $X_i\seq M_i$ for all $i\in I$. Note that any definable subset of $M$ is internal. An internal set $X\seq M$ is \textbf{pseudofinite} if  $X=\prod_{\cU}X_i$ where each $X_i$ is finite. 

Given a sequence $(r_i)_{\in I}$ in the extended nonnegative real line $R:=\R_{\geq 0}\cup\{\infty\}$, there is a unique $s\in R$ such that $\lim_{\cU}r_i=s$, i.e.,  $\{i\in I:|r_i-s|<\epsilon\}\in\cU$ for all $\epsilon>0$ (here we set $\infty-\infty=0$). If $X\seq M$ is nonempty, internal, and pseudofinite, then we define the \textbf{$X$-normalized pseudofinite counting measure}, denoted $\mu_X$, such that, given an internal set $A=\prod_{\cU}A_i$, $\mu_X(A)=\lim_{\cU}|A_i|/|X_i|$. Then $\mu_X$ is an extended Keisler measure over the Boolean algebra of internal subsets of $M$. For the subsequent results, we will need to lift $\mu_X$  to elementary extensions of $M$. To do this, we expand $M$ by a sort for $(R,+,<)$ and, for any formula $\phi(x;\ybar)$ and any pseudofinite definable set $X\seq M$, a function $f^X_\phi$ from the $\ybar$ sort in $M$ to the $R$ sort, which is interpreted as $\bbar\mapsto \mu_X(\phi(M;\bbar))$. From now on, we view $M$ as a structure in this expanded language.\footnote{To avoid ambiguity in words like ``definable", one can iterate the expansion by the symbols $f^X_\phi$.}

Now let $M^*$ be an elementary extension of $M$ in this expanded language. Then in the $R$ sort we have an elementary extension $R^*$ of $R$, which can be equipped with the usual \emph{standard part map}, denoted $\st$, on elements in $R^*$ of finite absolute value. We extend $\st$ to all of $R^*$ by setting $\st(x)=\infty$ for any infinite $x\in R^*$. 

Given $X\seq M^*$, we let $X(M)$ denote $X\cap M$.
 A definable set $X\seq M^*$ is \textbf{pseudofinite} if $X$ is definable over $M$ and $X(M)$ is pseudofinite. Given a pseudofinite definable set $X\seq M^*$, we define the \textbf{$X$-normalized pseudofinite counting measure} $\mu_X$ on $\Def(M^*)$ such that, given a formula $\phi(x;\ybar)$ and $\bbar\in (M^*)^{\ybar}$, $\mu_X(\phi(M^*;\bbar))=\st(f^{X(M)}_\phi(\bbar))$. Then $\mu_X$ is an extended $\Def(M^*)$-Keisler measure, and $\mu_X(X)=1$. Moreover, $\mu_X(A)=\mu_X(A(M))$ for any $A\seq M^*$ definable over $M$ (where $\mu_X(A(M))$ is computed in $M$). This is one of the reasons why we will take some care to keep track of definability of certain objects over $M$. 

For the rest of the paper, when we say that a structure  is sufficiently saturated, we will work in this expanded language for pseudofinite counting measures, and assume that $M^*$ is an elementary extension of an ultraproduct $M$ as above.  When considering expansions of groups, we will use $G^*$ and $G$ instead of $M^*$ and $M$.

The next proposition describes the definability and finite satisfiability properties of the measures $\mu_X$ with respect to NIP formulas.

\begin{proposition}\label{prop:DFS}
Suppose $\phi(x;\ybar)$ is an NIP formula, and $X\seq M^*$ is definable and pseudofinite. Let $Y=\{\bbar\in (M^*)^{\ybar}:\phi(M^*;\bbar)\seq X\}$,  and fix $\epsilon>0$. 
\begin{enumerate}[$(a)$]
\item There is a finite set $F\seq X$ such that, if $\bbar\in Y$ and $\mu_X(\phi(x;\bbar))>\epsilon$, then $\phi(M^*;\bbar)\cap F\neq\emptyset$.
\item There is a set $D\seq (M^*)^{\ybar}$, which is countably $\phi^*$-type-definable over $X$, such that, given $\bbar\in Y$, $\mu_X(\phi(x;\bbar))\leq\epsilon$ if and only if $\bbar\in D$.
\end{enumerate}
If, moreover, $\phi(x;\ybar)$ is over $M$, then in part $(a)$ we may assume $F\seq X(M)$, and in part $(b)$ we may assume that $D$ is countably $\phi^*$-type-definable over $X(M)$.
\end{proposition}
\begin{proof}
This is a consequence of Theorem \ref{thm:VC} and {\L}o\'{s}'s Theorem. (See also \cite[Section 2]{CPpfNIP}, which deals with the case that $M^*$ is pseudofinite and $X=M^*$.)
\end{proof}

Finally, we state a pseudofinite analogue of Theorem \ref{thm:Mat}.

\begin{proposition}\label{prop:MatPF}
Let $\phi(x;\ybar)$ be a $d$-NIP formula. Suppose $X\seq M^*$ is definable and pseudofinite, and $Y\seq (M^*)^{\ybar}$ is a definable set such that $\phi(M^*,\bbar)\seq X$ for all $\bbar\in Y$. Fix $p\geq q\geq 2^{d+1}$ and suppose that for any $\bbar_1,\ldots,\bbar_p\in Y$ there is $I\seq[p]$ with $|I|=q$ such that $\bigcap_{i\in I}\phi(M^*;\bbar_i)\neq\emptyset$. Then there is a finite set $F\seq X$ such that, for any $\bbar\in Y$, $\phi(M^*;\bbar)\cap F\neq\emptyset$.
\end{proposition}
\begin{proof} This can be seen by applying Theorem \ref{thm:Mat} to the finite sets $X_{i}$ (of which $X$ is the ultraproduct) and then again applying {\L}o\'{s}'s Theorem.

\end{proof}

\section{The NIP stabilizer theorem}\label{sec:NIPLie}

In this section we let $G^*$ be a sufficiently saturated first-order structure expanding a group. So we work over a fixed ultraproduct $G=\prod_{\cU}G_i\prec G^*$, where each $G_i$ is an expansion of a group in a common language.
The next definition provides analogues of some notions from Section \ref{sec:Gprelim1}. 

\begin{definition}\label{def:pftrip}
Suppose $A\seq G^*$ is definable.
\begin{enumerate}
\item $A$ has \textbf{finite tripling} if it is pseudofinite and $\mu_A(A^3)<\infty$.
\item $A$ is an \textbf{approximate group} if $A$ is symmetric  and $A^2$ can be covered by finitely many left translates of $A$.
\end{enumerate}
\end{definition}

In our earlier definition of a subset $A$ of a group $G$ having $k$-tripling (Definition 3.1) there was an inbuilt assumption that $A$ was finite. As we have now moved to a ``pseudofinite setting", this is reflected in the previous definition with the inbuilt assumption that $A$ is pseudofinite (but possibly infinite). So the reader should be aware of this ambiguity, which can always be clarified by taking note of the surrounding context.

Note that a symmetric definable set $A\seq G^*$ is an approximate group if and only if $A$ is left generic in $\langle A\rangle$. 
Recall also that if $A\seq G^*$ is definable and pseudofinite then $A(G)=\prod_{\cU}A_i$, where each $A_i\seq G_i$ is finite. In this case, one can check that $A$ has finite tripling if and only if there is some $k\geq 1$ such that the set of $i\in I$, for which $A_i$ has $k$-tripling, is in $\cU$. Similarly, $A$ is an approximate group if and only if there is some $k\geq 1$ such that the set of $i\in I$, for which $A_i$ is a $k$-approximate group, is in $\cU$. Applying Theorem \ref{thm:TaoPSE}, we obtain the following conclusion.

\begin{corollary}\label{cor:TaoPSE}
Suppose $A\seq G^*$ is definable and pseudofinite with finite tripling. Then $\mu_A(A^{\pm n})<\infty$ for all $n\geq 1$, and $A^{\pm 2}$ is an approximate group. 
\end{corollary}

Consequently, a pseudofinite definable set $A\seq G^*$ with finite tripling gives rise to a ``locally pseudofinite" ind-definable subgroup $\langle A\rangle$. In particular, if $X\seq \langle A\rangle$ is definable then, by saturation, $X\seq A^{\pm n}$ for some $n\geq 1$, and thus $\mu_A(X)<\infty$ (which implies $X$ is pseudofinite if $X$ is definable over $G$).

In this setting, a simplified version of Hrushovski's stabilizer theorem from \cite{HruAG} is as follows: If $A\seq G^*$ is definable and pseudofinite with finite tripling, then  there is a type-definable bounded-index normal subgroup $\Gamma$ of $\langle A\rangle$ such that $\Gamma\seq A^{\pm 8}$. Note that this statement can be viewed as a pseudofinite Freiman-Ruzsa property, in the same sense as discussed at the start of Section \ref{sec:intro2}. Hrushovski's approach is based the general theory of stabilizers of types in definable groups, and ``type-definability" of $\Gamma$ is in the sense of a similarly expanded language for measures. In \cite{BGT}, Breuillard, Green, and Tao provide a different proof based on a combinatorial argument of Sanders \cite{SanBS}. This approach was later modified by Massicot and Wagner \cite{MassWa} in order to remove the necessity of expanding the language. In \cite{CoBogo}, the first author used similar modifications to obtain $\Gamma\seq A^2A^{\nv 2}\cap A^{\nv 2}A^2\cap (AA\inv)^2\cap (A\inv A)^2$. 

The goal of this section is to prove a stronger version of the above statement under the assumption that $A$ is NIP. In particular, we will exhibit stronger definability properties for $\Gamma$, and also that $\Gamma\seq AA\inv\cap A\inv A$. 

\begin{definition}
Suppose $A\seq G^*$ is definable and pseudofinite. Given $\epsilon\geq 0$ and a definable set $X\seq G^*$, define 
\begin{align*}
\Stab_{\!A,\epsilon}^\ell(X) &= \{g\in G^*:\mu_A(gX\smd X)\leq \epsilon\},\text{ and}\\
\Stab_{\!A,\epsilon}^r(X) &= \{g\in G^*:\mu_A(Xg\smd X)\leq\epsilon\}.
\end{align*}
Set $\Stab^\ell_{\!A}(X)=\Stab^\ell_{\!A,0}(X)$ and $\Stab^r_{\!A}(X)=\Stab^r_{\!A,0}(X)$. 
\end{definition}

\begin{proposition}\label{prop:mstab}
Suppose $A,X\seq G^*$ are definable, and $A$ is pseudofinite.
\begin{enumerate}[$(a)$]
\item For any $\epsilon\geq 0$, $\Stab^\ell_{\!A,\epsilon}(X)$ and $\Stab^r_{\!A,\epsilon}(X)$ are symmetric subsets of $G^*$.
\item $\Stab^\ell_{\!A}(X)$ and $\Stab^r_{\!A}(X)$ are subgroups of $G^*$.
\item If $0\leq\epsilon<\mu_A(X)$ then $\Stab^\ell_{\!A,\epsilon}(X)\seq XX\inv$ and $\Stab^r_{\!A,\epsilon}(X)\seq X\inv X$.
\end{enumerate}
\end{proposition}
\begin{proof}
Part $(a)$ follows from bi-invariance of $\mu_A$, and part $(b)$ then follows from finite additivity. For part $(c)$, note that if $\mu_A(gX\smd X)<\mu_A(X)$ then $\mu_A(gX\cap X)>0$, and so $gX\cap X\neq\emptyset$, i.e. $g\in XX\inv$. The other inclusion is similar.
\end{proof}

Next we investigate the behavior of these stabilizer subgroups when the set $A$ in question is NIP. First, we note a pseudofinite analogue of Theorem \ref{thm:TaoNIP}.

\begin{corollary}\label{cor:TaoNIP}
Suppose $A\seq G^*$ is definable, NIP, and pseudofinite with finite tripling. Then for all $m\geq 2$ there are finite $E_m,F_m\seq A(G)^{\pm(m+1)}$ such that $A^{\pm m}\seq E_mA\cap AF_m$. In particular, $A^{\pm 1}$ is an approximate group.
\end{corollary}
\begin{proof}
For $m=2$, transfer Theorem \ref{thm:TaoNIP} to the pseudofinite setting. Then proceed by induction. 
\end{proof}

Given a pseudofinite definable set $A\seq G^*$, we say that $X\seq \langle A\rangle$ is \textbf{$\cR_{\!A}$-definable over $G$} if $X\in\cR_{\!A}$ and $X$ is definable over $G$. Note that $X$ is $\cR_{\!A}$-definable over $G$ if and only if it is in the ring generated by $\{gAh:g,h\in \langle A(G)\rangle\}$. 
We formulate analogous notions with $\cR_{\!A}$ replaced by $\cR^{\ell}_{\!A}$ or $\cR^r_{\!A}$ in the obvious way. A consequence of the previous corollary is that if $A\seq G^*$ is definable, NIP, and pseudofinite with finite tripling, then $\langle A\rangle$ is countably $\cR^\ell_{\!A}$-ind-definable over $G$ and countably $\cR^r_{\!A}$-ind-definable over $G$.

\begin{lemma}\label{lem:NIPLie}
Suppose $A\seq G^*$ is definable, NIP, and pseudofinite with finite tripling. Then $\Stab^\ell_{\!A}(A)$ and $\Stab^r_{\!A}$ are bounded-index subgroups of $\langle A\rangle$. Moreover, $\Stab^\ell_{\!A}(A)$ is countably $\cR^r_{\!A}$-type-definable over $G$, and $\Stab^r_{\!A}(A)$ is countably $\cR^\ell_{\!A}$-type-definable over $G$.
\end{lemma}
\begin{proof}
We just consider $\Stab^\ell_{\!A}(A)$ (the proof for $\Stab^r_{\!A}(A)$ is similar). Note first that $\Stab^\ell_{\!A}(A)=\bigcap_{\epsilon>0}\Stab^\ell_{\!A,\epsilon}(A)$, and so it suffices to fix $0<\epsilon<1$, and show that $S:=\Stab^\ell_{\!A,\epsilon}(A)$ is countably $\cR^r_{\!A}$-type-definable over $G$ and left generic in $\langle A\rangle$.

For any $m\geq 1$, $A^{\pm m}$ is definable and pseudofinite. If we set $t_m=\mu_A(A^{\pm m})$ then $1\leq t_m<\infty$ by Corollary \ref{cor:TaoPSE}, and $\mu_A=t_m\mu_{A^{\pm m}}$.  Let $\theta(x;y)$ denote the formula $A(y\cdot x)\smd A(x)$, and note that $\theta(x;y)$ is NIP. By Corollary \ref{cor:TaoNIP}, there is a finite set $F\seq A(G)^{\pm 3}$ such that $A^{\pm 2}\seq AF$. So $S\seq AF$ by Proposition \ref{prop:mstab}$(c)$ and since $\epsilon<1$. Set $Y=\{b\in G^*:\theta(G^*;b)\seq A^{\pm 5}\}$. Then $S\seq A^{\pm 2}\seq AF\seq A^{\pm 4}\seq Y$, and so $S=\{g\in AF:\mu_{A^{\pm 5}}(\theta(x;g))\leq\epsilon/t_5\}$. By Proposition \ref{prop:DFS}$(b)$, $S=D\cap AF$ for some set $D\seq G^*$, which is countably $\theta^*$-type-definable over $A(G)^{\pm 5}$. Note that any set, which is $\theta^*$-definable over $A(G)^{\pm 5}$, is a Boolean combination of right translates of $A$ by elements in $A(G)^{\pm 5}$. So $S$ is countably $\cR^r_{\!A}$-type-definable over $G$.

Finally, we fix $m\geq 1$ and show that $X:=A^{\pm m}$ can be covered by finitely many left translates of $S$. Let $\phi(x;y_1,y_2)$ denote the formula $A(y_1\inv \cdot x)\smd A(y_2\inv \cdot x)$. Then $\phi(x;y_1,y_2)$ is NIP, and if $b_1,b_2\in A^{\pm m}$ then $\phi(G^*;b_1,b_2)\seq X$. By Proposition \ref{prop:DFS}$(a)$, there is a finite set $F\seq X$ such that if $b_1,b_2\in A^{\pm m}$ and $\mu_X(\phi(x;b_1,b_2))>\epsilon/t_{m+1}$, then $\phi(G^*;b_1,b_2)\cap F\neq\emptyset$. Define an equivalence relation $\sim$ on $A^{\pm m}$ such that $g\sim h$ if $F\cap gA=F\cap hA$. Since $F$ is finite, we can choose representatives $g_1,\ldots,g_n\in A^{\pm m}$ of all $\sim$-classes in $A^{\pm m}$. Now fix $b\in A^{\pm m}$. Then $b\sim g_i$ for some $i\leq n$, which implies $\phi(G^*;b,g_i)\cap F=\emptyset$. So $\mu_A(g_i\inv bA\smd A)=\mu_A(bA\smd g_iA)=t_{m+1}\mu_X(\phi(x;b,g_i))\leq \epsilon$, and thus we have $g_i\inv b\in S$, i.e., $b\in g_iS$. Altogether, $A^{\pm m}\seq g_1S\cup\ldots\cup g_nS$.
\end{proof}

\begin{definition}
Given a pseudofinite definable set $A\seq G^*$, define 
\[
\Gamma^\ell_{\!A} = \textstyle\bigcap_{g\in \langle A\rangle}\Stab_{\!A}^\ell(gA)\makebox[.5in]{and} \Gamma^r_{\!A} = \textstyle\bigcap_{g\in \langle A\rangle}\Stab_{\!A}^r(Ag).
\]
\end{definition}

We now consider stabilizers of types. Note that if $A\seq G^*$ is definable, then $\langle A\rangle$ acts by left and right translation on the space $S(\cR_{\!A})$ of $\cR_{\!A}$-types. 

\begin{definition}
Suppose $A\seq G^*$ is definable and pseudofinite. Given $p\in S(\cR_{\!A})$, define the stabilizers
\[
\Stab^\ell_{\!A}(p)=\{g\in \langle A\rangle:gp=p\}\makebox[.5in]{and} \Stab^r_{\!A}(p)=\{g\in \langle A\rangle :pg=p\}.
\]
\end{definition}

\begin{lemma}\label{lem:Stabs}
Suppose $A\seq  G^*$ is definable and pseudofinite. Fix $p\in S(\cR_{\!A})$ and  $\star\in\{\ell,r\}$.
\begin{enumerate}[$(a)$]
\item If $p$ is $\mu_A$-wide then $\Gamma^\star_{\!A}\leq \Stab^\star_{\!A}(p)$. 
\item $\Stab^\star_{\!A}(p)$ is contained in any $\cR_{\!A}$-type-definable bounded-index subgroup of $\langle A\rangle$.
\end{enumerate}
\end{lemma}
\begin{proof}
We consider the case when $\star$ is $\ell$ (the other case is  similar).

Part $(a)$. Assume $p$ is $\mu_A$-wide and fix $x\in\Gamma^\ell_{\!A}$. For a contradiction, suppose $x\not\in\Stab^\ell_{\!A}(p)$. Since $x\in\langle A\rangle$ (by Proposition \ref{prop:mstab}$(c)$), we have $xp\neq p$, and so there are  $g,h\in \langle A\rangle$ such that $xgAh\smd gAh\in p$. So  $\mu_A(xgAh\smd gAh)>0$, i.e., $\mu_A(xgA\smd gA)>0$, and thus $x\not\in \Stab^\ell_{\!A}(gA)$, which contradicts $x\in \Gamma^\ell_{\!A}$.

Part $(b)$. Let $\Gamma$ be an $\cR_{\!A}$-type-definable bounded-index subgroup of $\langle A\rangle$. Let $C$ be the unique right coset of $\Gamma$ such that $p\models C$. Then for any $x\in\Stab^\ell_{\!A}(p)$, we have $p\models xC\cap C$, which implies $xC\cap C\neq\emptyset$, and thus $x\in\Gamma$.
\end{proof}

We can now state and prove the main result of this section.

\begin{theorem}\label{thm:Stabs}
Suppose $A\seq G^*$ is definable, NIP, and pseudofinite with finite tripling. 
\begin{enumerate}[$(a)$]
\item $\langle A\rangle$ has a smallest $\cR_{\!A}$-type-definable subgroup of bounded-index, denoted $\Gamma_{\!A}$. 
\item $\Gamma_{\!A}$ is countably $\cR_{\!A}$-type-definable over $G$ and normal in $\langle A\rangle$.
\item $\Gamma_{\!A}=\Gamma^\ell_{\!A}=\Gamma^r_{\!A}$ and $\Gamma_{\!A}\seq AA\inv\cap A\inv A$.
\item If $p\in S(\cR_{\!A})$ is $\mu_A$-wide, then $\Gamma_{\!A}=\Stab^\ell_{\!A}(p)=\Stab^r_{\!A}(p)$. 
\end{enumerate}
\end{theorem}
\begin{proof}
We first show that $\Gamma^\ell_{\!A}$ is an $\cR_{\!A}$-type-definable bounded-index normal subgroup of $\langle A\rangle$, which is contained in $AA\inv$. Note first that $\Gamma^\ell_{\!A}\seq AA\inv$ by  Proposition \ref{prop:mstab}$(c)$. Let $H=\Stab^{\ell}_{\!A}(A)$. If $g\in G^*$ then $\Stab^\ell_{\!A}(gA)=gHg\inv$, and so $\Gamma^\ell_{\!A}=\bigcap_{g\in\langle A\rangle}gHg\inv$. So it follows from Lemma \ref{lem:NIPLie} that $\Gamma^\ell_{\!A}$ is an $\cR_{\!A}$-type-definable bounded-index normal subgroup of  $\langle A\rangle$.

By a similar argument, $\Gamma^r_{\!A}$ is an $\cR_{\!A}$-type-definable bounded-index normal subgroup of $\langle A\rangle$, which is contained in $A\inv A$. 

Now let $p\in S(\cR_{\!A})$ be $\mu_A$-wide (note that such types exist by Fact \ref{fact:wide}). By Lemma \ref{lem:Stabs}, we have
\[
\Gamma^\ell_{\!A}\leq \Stab^\ell_{\!A}(p)\leq\Gamma^r_{\!A}\leq\Stab^r_{\!A}(p)\leq\Gamma^\ell_{\!A},
\]
and so $\Gamma^\ell_{\!A}$, $\Gamma^r_{\!A}$, $\Stab^\ell_{\!A}(p)$, and $\Stab^r_{\!A}(p)$ are all the same group. It also follows from Lemma \ref{lem:Stabs} that this group is the smallest $\cR_{\!A}$-type-definable bounded-index subgroup of $\langle A\rangle$. Thus we have the group $\Gamma_{\!A}$ described in part $(a)$. It remains to show that $\Gamma_{\!A}$ is countably $\cR_{\!A}$-type-definable over $G$. For this, we first note that, by part $(a)$, $\Gamma_{\!A}$ is automorphism invariant with respect to the group language expanded by a formula defining $A$. It follows that $\Gamma_{\!A}$ is countably type-definable over $\emptyset$. A standard saturation argument then shows that $\Gamma_{\!A}$ is countably $\cR_{\!A}$-type-definable over $G$ (indeed, over any model $M\prec G^*$). 
\end{proof}

One could also denote $\Gamma_{\!A}$ by ${\langle A\rangle}^{00}_{\cR_{\!A}}$ (the type-definable connected component of ${\langle A\rangle}$ with respect to the ring $\cR_{\!A}$ of formulas, or definable sets). 

\begin{remark}\label{rem:Stabs}
It follows from the proof of the previous theorem that if $A\seq G^*$ is definable, NIP, and pseudofinite with finite tripling, then $\Gamma_{\!A}$ coincides with $\bigcap_{g\in I}g\Stab^\ell_{\!A}(A)g\inv$ and $\bigcap_{g\in I}g\Stab^r_{\!A}(A)g\inv$ for some countable $I\seq \langle A(G)\rangle$.

If $\langle A\rangle$ is abelian, then we can further conclude $\Gamma_{\!A}=\Stab^\ell_{\!A}(A)=\Stab^r_{\!A}(A)$. But this can fail in general. For example, suppose $G^*$ is pseudofinite and $H$ is a non-normal finite-index definable subgroup of $G^*$. Let $A=aH$ where $a\in G^*$ is such that $aHa\inv\neq H$. Then $A$ is NIP and has finite tripling, but $\Stab^\ell_{\!A}(A)=aHa\inv$ and $\Stab^r_{\!A}(A)=H$. 
\end{remark}

\section{Measures and domination}\label{sec:dom}

The goal of this section is to prove Theorem \ref{thm:GCD}. Roughly speaking, this result says that if $G^*$ is a sufficiently saturated group and $A\seq G^*$ is definable, NIP, and pseudofinite with finite tripling, then the collection of cosets $C$ of $\Gamma_{\!A}$ in $\langle A\rangle$, such that both $C\cap A$ and $C\backslash A$ are $\mu_A$-wide, is Haar null as a subset of the locally compact group $\langle A\rangle/\Gamma_{\!A}$. This will be the origin for condition $(iii)$ of our main result (Theorem \ref{thm:NIPgen}).

In the special case that $\langle A\rangle$ is \emph{definable} (which, in particular, implies that $\langle A\rangle$ is pseudofinite and $\langle A\rangle/\Gamma_{\!A}$ is compact), the analogue of Theorem \ref{thm:GCD} was proved by the authors in \cite{CPpfNIP}. The proof in this case will follow a similar strategy, and most of the work will involve adapting various tools and techniques to the setting of locally compact groups.

\subsection{Measures}

Let $G^*\succ G=\prod_{\cU}G_i$ be a sufficiently saturated group. Note that if $A\seq G^*$ is definable and pseudofinite with finite tripling, then the $A$-normalized pseudofinite counting measure $\mu_A$ is a bi-invariant Keisler measure on $\Def(\langle A\rangle)$. In particular, the finite tripling assumption ensures that $\mu_A$ is $\R_{\geq 0}$-valued on $\Def(\langle A\rangle)$.

The first main result of this section (Lemma \ref{lem:null} below) will require the following standard Borel-Cantelli-type result about probability measures.

\begin{fact}\label{fact:BC}
For any $\epsilon>0$ and $q\geq 1$, there are $\delta>0$ and $p\geq 1$ such that the following holds. Let $B$ be a Boolean algebra and $\nu$ be a finitely additive probability measure on $B$. Suppose $x_1,\ldots,x_p\in B$ are such that $\nu(x_i)\geq\epsilon$ for all $i\in[p]$. Then there is a $q$-element set $I\seq[p]$ such that $\nu(\bigwedge_{i\in I}x_i)\geq\delta$.
\end{fact}

For the rest of this section, we fix a definable, NIP, pseudofinite set $A\seq G^*$ with finite tripling. 

\begin{lemma}\label{lem:null}
Given $X\in\cR_{\!A}$, the following are equivalent:
\begin{enumerate}[$(i)$]
\item $X$ is left generic in $\langle A\rangle$,
\item $X$ is right generic in $\langle A\rangle$,
\item $\mu_A(X)>0$.
\end{enumerate}
So a type $p\in S(\cR_{\!A})$ is $\mu_A$-wide if and only if every $X\in p$ is (left) generic in $\langle A\rangle$.
\end{lemma}
\begin{proof}
Note that $(i)\Rightarrow(iii)$ and $(ii)\Rightarrow(iii)$ follow immediately from finite additivity and invariance of  $\mu_A$. 

$(iii)\Rightarrow (i)$. The argument is similar to Theorem \ref{thm:TaoNIP}. Suppose $\mu_A(X)>0$. Fix $m\geq 1$. We want to find a finite set $F\seq\langle A\rangle$ such that $A^{\pm m}\seq FX$. Without loss of generality, we can assume $X\seq A^{\pm m}$. Since $A$ has finite tripling, we have $\mu_{AA^{\pm m}}(A)>0$ and so, for any $b\in G^*$, 
\[
\mu_{AA^{\pm m}}(Xb)=\mu_{AA^{\pm m}}(X)=\mu_A(X)\mu_{AA^{\pm m}}(A)>0.
\] 
Let $\phi(x;y)$ be the formula $X(x\cdot y)$. Then $\phi(x;y)$ is NIP and, if $b\in A^{\pm m}$, then $\phi(G^*,b)\seq AA^{\pm m}$. By Proposition \ref{prop:DFS}$(a)$, there is a finite set $F\seq AA^{\pm m}$ such that if $b\in A^{\pm m}$ then $Xb\inv\cap F\neq\emptyset$. So $A^{\pm m}\seq F\inv X$.

$(iii)\Rightarrow (ii)$. One can use a similar strategy as in $(iii)\Rightarrow (i)$, except applied to the formula $X(y \cdot x)$. 
\end{proof}


We call  a set $X\in\cR_{\!A}$ \textbf{generic} if it satisfies the equivalent conditions of the previous corollary. We will also refer to $\mu_A$-wide types in $S(\cR_{\!A})$ as \textbf{generic}, and let $S^g(\cR_{\!A})$ denote the set of generic $\cR_{\!A}$-types.

\begin{remark}\label{rem:genclosure}
The set $S^g(\cR_{\!A})$ is (topologically) closed in $S(\cR_{\!A})$, since any non-generic $\cR_{\!A}$-type contains a non-generic set in $\cR_{\!A}$. In fact, by modifying the usual argument for Boolean algebras (e.g., \cite[Lemma 1.7]{NewTD}), one can show that if $p\in S^g(\cR_{\!A})$ then $S^g(\cR_{\!A})$ is the closure of the orbit $\{gp:g\in\langle A\rangle\}$. 
\end{remark}

 Let $\G_{\!A}=\langle A\rangle/\Gamma_{\!A}$, and let $\pi\colon \langle A\rangle\to \G_{\!A}$ be the quotient map. 
 Recall that $\G_{\!A}$ is a second countable locally compact Hausdorff group under the logic topology defined in Section \ref{sec:LCLT}.

\begin{definition}
Given $X\in\cR_{\!A}$ and $p\in S^g(\cR_{\!A})$, define 
\[
\B^p_X=\{a\Gamma_{\!A}\in\G_{\!A}:X\in ap\}.
\]
Note that $\B^p_X$ is well-defined since $\Gamma_{\!A}=\Stab^\ell_{\!A}(p)$ by Theorem \ref{thm:Stabs}$(d)$. 
\end{definition}

\begin{lemma}\label{lem:Borel}
If $X\in\cR_{\!A}$ and $p\in S^g(\cR_{\!A})$, then $\B^p_X$ is $F_\sigma$ and $G_\delta$ in $\G_{\!A}$.
\end{lemma}
\begin{proof}
We first note that it suffices to show that $\B^p_X$ is $F_\sigma$ for all $X\in\cR_{\!A}$ and $p\in S^g(\cR_{\!A})$. Indeed, suppose this holds. By Corollary \ref{cor:TaoNIP}, we can write $\langle A\rangle=\bigcup_{m=1}^\infty Y_m$, where $Y_m\in\cR_{\!A}$. Now, given $X\in\cR_{\!A}$ and $p\in S^g(\cR_{\!A})$, we have $\G_{\!A}\backslash \B^p_X=\bigcup_{m=1}^\infty \B^p_{Y_m\backslash X}$, which is $F_\sigma$, and so $\B^p_X$ is $G_\delta$.

Now fix $X\in\cR_{\!A}$ and $p\in S^g(\cR_{\!A})$. Let $D\in p$ be arbitrary. If $a\in\pi\inv(\B^p_X)$ then $p\models a\inv X\cap D$, and so we have $\pi\inv(\B^p_X)\seq Y:=XD\inv$.  Let $Z=Y\cup  Y\inv X$, and note that $Z$ is definable.  By Proposition \ref{prop:DFS}$(a)$, there is a countable set $E\subset G^*$ such that for any $\epsilon>0$, if $a_1,a_2\in Y$ and $\mu_A(a_1\inv X\backslash a_2\inv X)>0$, then $(a_1\inv X\backslash a_2\inv X)\cap E\neq\emptyset$. Without loss of generality, we may assume that $E\seq  Y\inv X\seq Z$.

We now define a new first-order structure $\cZ=(Z;U_Y,R_X)$ with universe $Z$, a unary relation $U_Y$ interpreted as $Y$, and a binary relation $R_X(x,y)$ interpreted as $X(x\cdot y)$. Note that $\cZ$ is interpretable in $G^*$. Working for now with respect to $\Th(\cZ)$, let $\phi(x;y)$ denote the NIP formula $R_X(y, x)\wedge U_Y(y)$. Then there is a complete $\phi$-type $p_0\in S_\phi(Z)$ such that, given $a\in Z$, $\phi(x;a)\in p_0$ if and only if $a\inv X\in p$. 
Let $M\prec \cZ$ be a countable submodel such that $E\seq M$. We claim that $p_0$ is $M$-invariant. For a contradiction, suppose we have $a_1,a_2\in Z$ such that $a_1\equiv_M a_2$ and $\phi(x;a_1),\neg\phi(x;a_2)\in p_0$. Then $a_1\in Y$, which implies $a_2\in Y$, and thus $a_1\inv X\backslash a_2\inv X\in p$. Since $p$ is generic, there is some $m\in E\seq M$ such that $m\in a_1\inv X\backslash a_2\inv X$. So $\cZ\models R_X(a_1,m)\wedge\neg R_X(a_2,m)$, which contradicts $a_1\equiv_M a_2$. 

Let $S=\{q\in S_y(M):\phi(x;a)\in p_0\text{ for some/any }a\models q\}$. By \cite[Theorem 0.1]{SimRC}, $S$ is an $F_\sigma$ set in $S_y(M)$. So if $W=\{a\in Z:a\models q\text{ for some }q\in S\}$, then we can write $W=\bigcup_{n=0}^\infty W_n$, where $W_n\seq Z$ is type-definable (over $M$). 
It is also straightforward to show that $W=\pi\inv(\B^p_X)$ (recall that $\pi\inv(\B^p_X)\seq Y$ by assumption).

Now we return to the structure $G^*$. Note that each set $W_n$ above is type-definable (with respect to $G^*$) since $X$, $Y$, and $Z$ are all definable. So each $\pi(W_n)\seq\G_{\!A}$ is closed by Fact \ref{fact:LCLT}$(c)$, and thus $\B^p_X=\bigcup_{n=0}^\infty \pi(W_n)$ is $F_\sigma$. 
\end{proof}

\begin{definition}
Let $\eta$ be a Haar measure on $\G_{\!A}$. Given $p\in S^g(\cR_{\!A})$ and $X\in\cR_{\!A}$, define $\eta_p(X)=\eta(\B^p_X)$. (This is well-defined by Lemma \ref{lem:Borel}.)
\end{definition}

\begin{proposition}\label{prop:Haar}
Let $\eta$ be a Haar measure on $\G_{\!A}$ and suppose $p\in S^g(\cR_{\!A})$. Then $\eta_p$ is a left-$\langle A\rangle$-invariant $\cR_{\!A}$-Keisler measure.
\end{proposition}
\begin{proof}
We clearly have $\eta_p(\emptyset)=0$. If $X,Y\in\cR_{\!A}$ then $\B^p_{X\cup Y}=\B^p_X\cup \B^p_Y$ and $\B^p_{X\cap Y}=\B^p_X\cap \B^p_Y$, and so finite additivity of $\eta_p$ is inherited from additivity of $\eta$. Also, if $X\in\cR_{\!A}$ and $g\in\langle A\rangle$ then $\B^p_{gX}=g\Gamma\cdot \B^p_X$, and so $\eta_p(gX)=\eta_p(X)$ by left-invariance of $\eta$. 

Next, we show that $\eta_p$ is $\R_{\geq 0}$-valued. First, fix $X\in \cR_{\!A}$, and let $Y\in p$ be arbitrary. Then (as in the proof of Lemma \ref{lem:Borel}) we have $\B^p_X\seq\pi(XY\inv)$, which is compact Fact \ref{fact:LCLT}$(c)$. 
Therefore $\eta_p(X)\leq\eta(\pi(XY\inv))<\infty$. 

Finally, we show that $\eta_p$ is nontrivial. Fix $X\in\cR_{\!A}$ such that $\Gamma_{\!A}\seq X$. Let $U=\{a\Gamma_{\!A}\in\G_{\!A}:a\Gamma_{\!A}\seq X\}$. Then $U$ is nonempty by assumption, and open by Fact \ref{fact:LCLT}$(d)$. So $\eta(U)>0$. Moreover, if $a\Gamma_{\!A}\in U$ then, since $ap\models a\Gamma_{\!A}$, it follows that $X\in ap$, i.e., $a\Gamma_{\!A}\in \B^p_X$. So $U\seq \B^p_X$, and thus we have $\eta_p(X)>0$. 
\end{proof}

\subsection{Generic locally compact domination}
Before returning to the setting of the previous section, we first state a result of Simon \cite{SimGCD}, which is central to the proof of Theorem \ref{thm:GCD}. 

\begin{definition}
Let $\G$ be a locally compact group and let $\eta$ be a Haar measure on $\G$. A Borel set $W\seq\G$ is \textbf{pointwise large} if $\eta(W\cap U)>0$ for any open set $U\seq \G$ such that $W\cap U\neq\emptyset$.
\end{definition}

Note that whether a Borel set is pointwise large is independent of the choice of Haar measure.
The next result is Theorem 3.7 of \cite{SimGCD}. 

\begin{theorem}[Simon \cite{SimGCD}]\label{thm:SimGCD}
Let $\G$ be a locally compact second countable group, and suppose $B\seq \G$ is an NIP set such that both $B$ and $\G\backslash B$ are $F_\sigma$ and pointwise large. Then $\partial B$ is Haar null.
\end{theorem}

We now continue to fix a sufficiently saturated group $G^*$ and an NIP pseudofinite definable set $A\seq G^*$ with small tripling. Let $\G_{\!A}$ and $\pi\colon\langle A\rangle \to \G_{\!A}$ be as above. Recall that we also have a well-defined continuous surjective  map $\tau\colon S(\cR_{\!A})\to \G_{\!A}$ such that $p\models\tau(p)$ for any $p\in S(\cR_{\!A})$.

\begin{definition}
Given $X\in\cR_{\!A}$, define $\E_X$ to be the set of cosets $a\Gamma_{\!A}\in\G_{\!A}$ such that both $\Gamma_{\!A}\cap X$ and $\Gamma_{\!A}\backslash X$ are $\mu_A$-wide.
\end{definition}

Note that, by Fact \ref{fact:wide}, a coset $a\Gamma_{\!A}$ is in $\E_X$ if and only if there are $q,q'\in S^g(\cR_{\!A})$ such that $q\models a\Gamma_{\!A}\cap X$ and $q'\models a\Gamma_{\!A}\backslash X$.

\begin{lemma}\label{lem:GCD}
Suppose $X\in\cR_{\!A}$ and $p\in S^g(\cR_{\!A})$ is such that $p\models\Gamma_{\!A}$.
\begin{enumerate}[$(a)$]
\item $\B^p_X$ is NIP in $\G_{\!A}$.
\item $\B^p_X$ and $\G_{\!A}\backslash \B^p_X$ are pointwise large.
\item $\E_X$ is compact.
\item $\E_X\seq\partial \B^p_X$.
\end{enumerate}
\end{lemma}
\begin{proof}
Part $(a)$. Fix $X\in\cR_{\!A}$ and $p\in S^g(\cR_{\!A})$. We want to bound the VC-dimension of $\cS=\{\B^p_{gX}:g\in\langle A\rangle\}$. So suppose $\cS$ shatters $\{a_1\Gamma_{\!A},\ldots,a_n\Gamma_{\!A}\}\seq \G_{\!A}$.  Then for any $I\seq[n]$ we can choose $g_I\in\langle A\rangle$ such that $a_i\Gamma_{\!A}\in \B^p_{g_IX}$ if and only if $i\in I$. In other words, for $i\in [n]$, we have $a_i\inv g_IX\in p$ if and only if $i\in I$. It follows that there is some $b\in \langle A\rangle$ such that, for $i\in [n]$, $b\in a_i\inv g_IX$ if and only if $i\in I$. So $\{gX:g\in \langle A\rangle\}$ shatters $\{a_1b,\ldots,a_nb\}$. Altogether, $\VC(\cS)\leq \VC(\{gX:g\in\langle A\rangle\})$, and so $\VC(\cS)<\infty$ since $X$ is NIP (by Corollary \ref{cor:VCG}).

Part $(b)$. We first consider $\B^p_X$.  Suppose $U\seq\G_{\!A}$ is open and $\B^p_X\cap U\neq \emptyset$. We show $\eta(\B^p_X\cap U)>0$. Fix $a\Gamma_{\!A}\in \B^p_X\cap U$. Let $K=\tau\inv(\{a\Gamma_{\!A}\})$ and $V=\tau\inv(U)$. Then $K$ is closed, and contained in $[Y]$ for any $Y\in\cR_{\!A}$ containing $a\Gamma_{\!A}$. In particular, $K$ is compact. Since $V$ is open, and $K\seq V$, there is some $Y\in\cR_{\!A}$ such that $K\seq [Y]\seq V$. Note that $X\in ap$ since $a\Gamma_{\!A}\in\B^p_X$, and $Y\in ap$ since $ap\models a\Gamma_{\!A}$. So $X\cap Y\in ap$. Since $ap$ is generic, we have $\eta_p(X\cap Y)>0$ by Proposition \ref{prop:Haar}, which implies $\eta(\B^p_X\cap \B^p_Y)>0$. Thus it suffices to show that $\B^p_Y\seq U$.  So fix $g\Gamma_{\!A}\in \B^p_Y$. Then $Y\in gp$, i.e., $gp\in [Y]\seq V=\tau\inv(U)$. So $g\Gamma_A=\tau(gp)\in U$.

Now, for $\G\backslash\B^p_X$, recall from the proof of Lemma \ref{lem:Borel} that $\G\backslash \B^p_X=\bigcup_{m=1}^\infty\B^p_{Y_m\backslash X}$ for some $Y_m\in\cR_{\!A}$. By the above argument, each $\B^p_{Y_m\backslash X}$ is pointwise large, which implies that $\G\backslash\B^p_X$ is pointwise large.

Part $(c)$.  Choose an arbitrary set $D\in\cR_{\!A}$ such that $X\Gamma_{\!A}\seq D$. We claim that 
\[
\E_X=\tau(S^g(\cR_{\!A})\cap [X])\cap \tau(S^g(\cR_{\!A})\cap[D\backslash X]).
\]
To see this, first suppose $a\Gamma_{\!A}=\tau(q)=\tau(q')$ for some $q\in S^g(\cR_{\!A})\cap [X]$ and $q'\in S^g(\cR_{\!A})\cap[D\backslash X]$.  Then $q\models a\Gamma_{\!A}\cap X$ and $q'\models a\Gamma_{\!A}\cap D\backslash X$. In particular, $a\Gamma_{\!A}\cap X\neq\emptyset$, which implies $a\Gamma_{\!A}\seq D$. So $a\Gamma_{\!A}\cap D\backslash X=a\Gamma_{\!A}\backslash X$, which yields $a\Gamma_{\!A}\in \E_X$. Conversely, if $a\Gamma_{\!A}\in \E_X$ then there are $q,q'\in S^g(\cR_{\!A})$ such that $q\models a\Gamma_{\!A}\cap X$ and $q'\models a\Gamma_{\!A}\backslash X$. This again implies that $a\Gamma_{\!A}\backslash X=a\Gamma_{\!A}\cap D\backslash X$, and  so $q\in S^g(\cR_{\!A})\cap[X]$, $q'\in S^g(\cR_{\!A})\cap [D\backslash X]$, and $a\Gamma_{\!A}=\tau(q)=\tau(q')$.

Now, since $S^g(\cR_{\!A})\cap [X]$ and $S^g(\cR_{\!A})\cap [D\backslash X]$ are both compact (recall Remark \ref{rem:genclosure}), and $\tau$ is continuous, it follows that $\E_X$ is compact. 

Part $(d)$. Fix $a\Gamma_{\!A}\in \E_X$, and let $U\seq\G_{\!A}$ be an open neighborhood of $a\Gamma_{\!A}$. Fix $q,q'\in S^g(\cR_{\!A})$ such that $q\models a\Gamma_{\!A}\cap X$ and $q'\models a\Gamma_{\!A}\backslash X$. So $q,q'\in\tau\inv(\{a\Gamma_{\!A}\})\seq \tau\inv(U)$, and thus $\tau\inv(U)\cap[X]$ is an open neighborhood of $q$ and $\tau\inv(U)\backslash [X]$ is an open neighborhood of $q'$.
By Remark \ref{rem:genclosure}, there are $g,g'\in\langle A\rangle$, such that $gp\in \tau\inv(U)\cap[X]$ and $g'p\in
\tau\inv(U)\backslash[X]$. Since $p\models\Gamma_{\!A}$, we have $g\Gamma_{\!A},g'\Gamma_{\!A}\in U$. Also, since $X\in 
gp$ and $X\not\in g'p$, we have $g\Gamma_{\!A}\in \B^p_X$ and $g'\Gamma_{\!A}\not\in \B^p_X$. So $U\cap \B^p_X\neq
\emptyset\neq U\backslash \B^p_X$ and, altogether, it follows that $a\Gamma_{\!A}\in\partial \B^p_X$.
\end{proof}

\begin{theorem}\label{thm:GCD}
If $X\in\cR_{\!A}$ then $\E_X$ is compact and Haar null.
\end{theorem}
\begin{proof}
$\E_X$ is compact by Lemma \ref{lem:GCD}$(c)$.
Fix $X\in\cR_{\!A}$. Choose $p\in S^g(\cR_{\!A})$ such that $p\models\Gamma$. Then $\partial \B^p_X$ is Haar null by Lemma \ref{lem:Borel}, Theorem \ref{thm:SimGCD}, and parts $(a)$ and $(b)$ of Lemma \ref{lem:GCD}. So $\E_X$ is Haar null by Lemma \ref{lem:GCD}$(d)$. 
\end{proof}

Toward applying Theorem \ref{thm:GCD} to prove Theorem \ref{thm:NIPgen}, the first step is to remove the topological aspects and, in particular, relate Haar measures in $\G_{\!A}$ to the pseudofinite measure $\mu_A$.

\begin{lemma}\label{lem:oneY}
Let $\eta$ be a Haar measure on $\G_{\!A}$. Then there is some $\lambda>0$ such that, if $K\seq\G_{\!A}$ is compact, then $\eta(K)=\inf\{\lambda\mu_A(X):X\in\cR_{\!A},~\pi\inv(K)\seq X\}$. 
\end{lemma}
\begin{proof}
Since $\mu_A$ is a left-$\langle A\rangle$-invariant finitely additive $\R_{\geq 0}$-valued measure on the ring $\cR_{\!A}$, it admits a unique left-$\langle A\rangle$-invariant $\R_{\geq 0}$-valued regular Borel measure $\tilde{\mu}_A$ on $S(\cR_{\!A})$ such that $\tilde{\mu}_A([X])=\mu_A(X)$ for any $X\in\cR_{\!A}$. Given a Borel set $W\seq\G_{\!A}$, define $\eta_A(W)=\tilde{\mu}_A(\tau\inv(W))$. Then it is straightforward to check that $\eta_A$ is a Haar measure on $\G_{\!A}$, and that $\eta_A(K)=\inf\{\mu_A(X):X\in\cR_{\!A},~\pi\inv(K)\seq X\}$ for any compact $K\seq\G_{\!A}$. So there is some $\lambda>0$ such that $\eta=\lambda\eta_A$, as desired.
\end{proof}

We can now reformulate Theorem \ref{thm:GCD} in terms of the behavior of $\mu_A$ with respect to sets in $\cR_{\!A}$. The proof is nearly the same as that of \cite[Lemma 2.16]{CPTNIP}.

\begin{corollary}\label{cor:oneY}
 Let $(W_n)_{n=0}^\infty$ be a decreasing sequence of definable subsets of $\langle A\rangle$, such that $\Gamma_{\!A}=\bigcap_{n=0}^\infty W_n$. Then for any $X\in\cR_{\!A}$ and any $\epsilon>0$, there is some $n\geq 0$ and $Z\in\cR_{\!A}$ such that $\mu_A(Z)<\epsilon$ and if $g\in G^*\backslash Z$ then $\mu_A(gW_n\cap X)=0$ or $\mu_A(gW_n\backslash A)=0$. Moreover, if  $X$ is definable over $G$ then we may assume that $Z$ is definable over $G$.
\end{corollary}
\begin{proof}
Fix $X\in\cR_{\!A}$ and $\epsilon>0$. By Theorem \ref{thm:GCD} and Lemma \ref{lem:oneY}, there is some $Z\in\cR_{\!A}$ such that $\pi\inv(\E_X)\seq Z$ and $\mu_A(Z)<\epsilon$. If, moreover, $X$ is definable over $G$, then $\pi\inv(\E_X)$ is invariant under automorphisms of $G^*$ fixing $G$ pointwise, and thus is $\cR_{\!A}$-type-definable over $G$. Therefore, if $X$ is definable over $G$ then we may further assume that $Z$ is $\cR_{\!A}$-definable over $G$.

Toward a contradiction, suppose that for all $n\geq 0$ there is some $a_n\in G^*\backslash Z$ such that $\mu_A(a_nW_n\cap X)>0$ and $\mu_A(a_nW_n\backslash X)>0$. Note that, for all $n\geq 0$, $a_nW_n\cap X\neq\emptyset$, and so $a_n\in XW_n\inv\seq Y:=XW_0\inv$. So $a_n\Gamma_{\!A}\in \pi(Y)$ for all $n\geq 0$.

Let $U=\{a\Gamma_{\!A}\in \G_{\!A}:a\Gamma_{\!A}\seq Z\}$, which is open in $\G_{\!A}$ by Fact \ref{fact:LCLT}$(d)$. Note that $\E_X\seq U$ and $\pi\inv(U)\seq Z$. Also, for any $n\geq 0$, we have $a_n\Gamma_{\!A}\not\in U$ since $a_n\not\in Z$. Since $\pi(Y)$ is compact, $U$ is open, and $\G_{\!A}$ is second countable, we may pass to a subsequence and assume that $(a_n\Gamma_{\!A})_{n=0}^\infty$ converges to some $a\Gamma_{\!A}\in \pi(Y)\backslash U$. In particular, $a\Gamma_{\!A}\not\in \E_X$.  

Now, following the same argument as the claim in the proof of \cite[Lemma 2.16]{CPTNIP}, one can show that for all $n\geq 0$, $\mu_A(aW_n\cap X)>0$ and $\mu_A(aW_n\backslash X)>0$.
By saturation, this yields $a\Gamma_X\in \E_X$, which is a contradiction.
\end{proof}

\subsection{Uniqueness of measure}\label{sec:unique}

Once again, we fix a sufficiently saturated group $G^*$ and an NIP pseudofinite definable set $A\seq G^*$ with finite tripling. In this section, we observe that  $\mu_A$ is the unique left $\langle A\rangle$-invariant $\cR_{\!A}$-Keisler measure up to scalar multiples. This will not be needed for the main results of the paper, but we include it as a nice application of generic locally compact domination.

\begin{theorem}\label{thm:unique}
Suppose $\nu$ is a left-$\langle A\rangle$-invariant $\cR_{\!A}$-Keisler measure. Then there is some $\lambda>0$ such that $\nu(X)=\lambda\mu_A(X)$ for all $X\in\cR_{\!A}$.
\end{theorem}
\begin{proof}
We first show that if $X\in\cR_{\!A}$ and $\nu(X)>0$, then $X$ is generic. Fix $m\geq 1$. We  find a finite set $F\seq\langle A\rangle$ such that $A^{\pm m}\seq XF$. Without loss of generality, we can assume $X\seq A^{\pm m}$. Let $Y=A^{\pm m}$, and note that $Y$ is definable. By Corollary \ref{cor:TaoNIP}, there is a set $\widetilde{X}$, which is $\cR_{\!A}$-definable over $G$, such that $A^{\pm 2m}\seq \widetilde{X}$. Note that $\widetilde{X}$ is pseudofinite. Let $\phi(x;y)$ be the formula $X(y\cdot x)$. Then $\phi(x;y)$ is NIP, and if $b\in Y$ then $\phi(G^*;b)\seq \widetilde{X}$. Fix $d\geq 1$ such that $\phi(x;y)$ is $d$-NIP, and let $q=2^{d+1}$.  By assumption, we have $0<\nu(X)\leq\nu(\widetilde{X})<\infty$. So we can define a finitely-additive probability measure $\tilde{\nu}$ on $\cR^{\widetilde{X}}$ such that $\tilde{\nu}(W)=\nu(W)/\nu(\widetilde{X})$. In particular, for any $g\in Y$, we have $\tilde{\nu}(gX)=\tilde{\nu}(X)>0$. By Fact \ref{fact:BC}, there is some integer $p\geq 1$  such that for any $g_1,\ldots,g_p\in Y$ there is $I\seq [p]$ such that $|I|=q$ and $\tilde{\nu}(\bigcap_{i\in I}g_iX)>0$. So $\phi(x;y)$, $\widetilde{X}$, and $Y$ satisfy the hypothesis of Proposition \ref{prop:MatPF}, witnessed by $p$ and $q$, and thus there is a finite set $E\seq \widetilde{X}$ such that, for any $g\in Y$, $g\inv X\cap E\neq\emptyset$. Setting $F=E\inv$, we have $A^{\pm m}=Y\seq XF$.

Now, as in the proof of Lemma \ref{lem:oneY}, we extend $\nu$ to a left-$\langle A\rangle$-invariant $\R_{\geq 0}$-valued regular Borel measure $\widetilde{\nu}$ on $S(\cR_{\!A})$ such that $\widetilde{\nu}([X])=\nu(X)$ for any $X\in\cR_{\!A}$. As before, this induces a Haar measure $\eta$ on $\G_A$ such that $\eta(W)=\widetilde{\nu}(\tau\inv(W))$ for any Borel $W\seq\G_A$. Let $S=S^g(\cR_{\!A})$. We show that for any $X\in\cR_{\!A}$, $\nu(X)=\eta(\tau(S\cap [X]))$, which implies the desired result. 

First, since $\nu(X)=0$ for any non-generic $X\in\cR_{\!A}$, and $\widetilde{\nu}$ is regular, it follows that $\widetilde{\nu}(S(\cR_{\!A})\backslash S)=0$, and so $\widetilde{\nu}(W)=\widetilde{\nu}(S\cap W)$ for any Borel $W\seq S(\cR_{\!A})$. Now fix $X\in\cR_{\!A}$. Let $W=\tau\inv(\tau(S\cap [X]))\backslash (S\cap [X])$. Then $\tau(S\cap Z)\seq\E_X$ and so, by Theorem \ref{thm:GCD}, we have $\widetilde{\nu}(Z)=\widetilde{\nu}(S\cap Z)\leq \eta(\tau(S\cap Z))=0$. Therefore
\[
\nu(X)=\widetilde{\nu}(S\cap X)+\widetilde{\nu}(Z)=\widetilde{\nu}(\tau\inv(\tau(S\cap [X])))=\eta(\tau(S\cap [X])).\qedhere 
\]
\end{proof}

\begin{remark}
Using a similar strategy, one can show that the conclusion of Theorem \ref{thm:unique} also holds for any \emph{right}-$\langle A\rangle$-invariant $\cR_{\!A}$-Keisler measure. 
\end{remark}

\section{Proof of the main results}\label{sec:proofs}

In this section we will prove Theorems \ref{thm:NIPgen} and \ref{thm:NIPexp}.  

\subsection{Main result: pseudofinite version}

Let $G=\prod_{\cU}G_i$, where each $G_i$ is a group and $\cU$ is a nonprincipal ultrafilter on some index set $I$. W call a subset $W$ of $G$ an \emph{internal approximate subgroup} if it is of the form $\prod_{\cU}W_i$ where $W_{i}$ is a $k$-approximate subgroup of $G_{i}$ for some fixed $k$.

The next result transfers Corollary \ref{cor:oneY} to $G$, and reformulates the regularity statement in terms of approximate groups.

\begin{lemma}\label{lem:UP1}
Suppose $A\seq G$ is  internal, NIP, and pseudofinite with finite tripling.
Then for any $\epsilon>0$, there is an $\cR_{\!A}$-definable internal approximate group $W\seq G$, and an $\cR_{\!A}$-definable set $Z\seq G$, such that $W^4\seq AA\inv\cap A\inv A$, $A$ can be covered by finitely many left translates of $W$, $\mu_A(Z)<\epsilon$, and if $g\in G\backslash Z$ then $\mu_A(gW^4\cap A)=0$ or $\mu_A(gW^4\backslash A)=0$.
\end{lemma}
\begin{proof}
View $G$ as a first-order structure in the group language expanded by a unary predicate naming $A$. Let $G^*\succ G$ be a sufficiently saturated elementary extension, and let $A_*=A(G^*)$.   By Theorem \ref{thm:Stabs}, we may choose a decreasing sequence $(W_n)_{n=0}^\infty$ of subsets of $A_*A_*\inv\cap A_*\inv A_*$ such that each $W_n$ is $\cR_{\!A_*}$-definable over $G$ and  $\Gamma_{\!A_*}=\bigcap_{n=0}^\infty W_n$. Since $(\Gamma_{\!A_*})^4=\Gamma_{\!A_*}$, it follows from saturation that $\Gamma_{\!A_*}=\bigcap_{n=0}^\infty W_n^4$.  For any $n\geq 0$, we have $\Gamma_{\!A_*}\seq W_n$, and so $W_n$ is generic in $\langle A_*\rangle$. It follows that for any $n\geq 0$, $W_n$ is an approximate group and $A_*$ can be covered by finitely many left translates of $W_n$. Now fix $\epsilon>0$. By Corollary \ref{cor:oneY}, there is $Z_*\seq G^*$ and $n\geq 0$ such that $Z_*$ is $\cR_{\!A}$-definable over $G$, $\mu_A(Z_*)<\epsilon$, and if $g\in G^*\backslash Z_*$ then $\mu_A(gW_n^4\cap A_*)=0$ or $\mu_A(gW_n^4\backslash A_*)=0$. 

Now let $Z=Z_*(G)$ and $W=W_n(G)$. Since $W_n$ and $Z_*$ are definable over $G$, we have the desired conclusions. 
\end{proof}

Next, we state a pseudofinite version of \emph{Ruzsa's Covering Lemma}.

\begin{lemma}\label{lem:RCL}
Suppose $X,Y\seq G$ are internal and pseudofinite, with $\mu_Y(XY)<\infty$. Then there is $F\seq X$ such that $|F|\leq\mu_Y(XY)$ and $X\seq FY^2$. 
\end{lemma}
\begin{proof}
See \cite[Lemma 5.1]{BGT} for the finite version, which implies the pseudofinite version via {\L}o\'{s}'s Theorem. (Or one can directly adapt the proof to pseudofinite counting measures, e.g., as in \cite[Proposition 4.9]{CPTNIP}.) 
\end{proof}

We are now ready to bring coset nilprogressions into the picture. 

\begin{definition}
A subset $P\seq G$ is an \textbf{internal coset nilprogression} if there are $r,s\in\N$ such that $P=\prod_{\cU}P_i$, where each $P_i$ is a coset nilprogression in $G_i$ of rank $r$ and step $s$. 
Moreover, $P$ is in \textbf{normal form} if each $P_i$ is in $c$-normal form for some fixed $c\geq 1$. 
\end{definition}

Now we state the pseudofinite version of the main structure theorem for approximate groups (see \cite[Theorem 4.2]{BGT}).  The paper \cite{BGT} uses the expressions ``ultra-approximate subgroups" and ``ultra-coset nilprogressions", but in our context of $G=\prod_{\cU}G_i$, and the definition above, we can say ``internal" in place of ``ultra".

\begin{theorem}[Breuillard, Green, \& Tao \cite{BGT}]\label{thm:BGTpf}
Suppose $W\seq G$ is an internal approximate group. Then there is an internal coset nilprogression $P\seq G$ in normal form  such that  $P\seq W^4$ and $\mu_P(W)<\infty$. 
\end{theorem}

\begin{remark}
We will apply the previous theorem in the setting where  $W$ is in fact NIP. Following the sketch in Section \ref{sec:sketch}, the ``stabilizer theorem step" of Theorem \ref{thm:BGTpf} (which is also called ``Sanders-Croot-Sisask theory" in \cite{BGT}) could be replaced by Theorem \ref{thm:Stabs} above. On the other hand, we still need  the Gleason-Yamabe steps and ``escape norm" technology underlying \cite[Theorem 4.2]{BGT} in order recover coset nilprogressions. 
\end{remark}

We can now prove a pseudofinite version of Theorem \ref{thm:NIPgen}. 

\begin{theorem}\label{thm:mainUP}
Suppose $A\seq G$ is internal, NIP, and pseudofinite with finite tripling. Then for any $\epsilon>0$, there is an internal coset nilprogression $P\seq G$ in normal form and an internal set $Z\seq AP$, with $\mu_A(Z)<\epsilon$, satisfying the following properties.
\begin{enumerate}[$(i)$]
\item $P\seq AA\inv\cap A\inv A$ and $A\seq CP$ for some finite $C\seq A$.
\item There is a set $D\seq C$ such that $\mu_A((A\smd DP)\backslash Z)=0$.
\item  If $g\in G\backslash Z$ then $\mu_A(gP\cap A)=0$ or $\mu_A(gP\backslash A)=0$.
\end{enumerate}
\end{theorem}
\begin{proof}
Fix $\epsilon>0$. By Lemma \ref{lem:UP1}, there is an internal approximate group $W\seq G$, and a set $Z\in\cR_{\!A}$, such that $W^4\seq AA\inv\cap A\inv A$, $A$ can be covered by finitely many left translates of $W$, $\mu_A(Z)<\epsilon$, and if $g\in G\backslash Z$ then $\mu_A(gW^4\cap A)=0$ or $\mu(gW^4\backslash A)=0$.  By Theorem \ref{thm:BGTpf}, there is an internal coset nilprogression $P\seq G$ in normal form such that $P\seq W^4$ and $\mu_P(W)<\infty$. In particular, since $P\seq W^4$, we have $\mu_A(gP\cap A)=0$ or $\mu_A(gP\backslash A)=0$ for any $g\in G\backslash Z$.

By assumption, we have $P=\prod_{\cU}P_i$ where each $P_i\seq G_i$ is a coset nilprogression of rank $r$  in $c$-normal form, for some fixed $c,r\geq 0$. For $i\in\N$, let $Q_i=P_i^{(1/2)}$, and set $Q=\prod_{\cU}Q_i$. Note that $Q^2\seq P$. By Proposition \ref{prop:scaleCN}, we have $|P_i|\leq O_{r,c}(|Q_i|)$ for all $i$, and so $\mu_Q(P)<\infty$. 

We claim that $\mu_Q(AQ)<\infty$. Indeed, we have shown $\mu_Q(P)<\infty$, and we have $\mu_P(W)<\infty$ by assumption.  Also, $\mu_W(A)<\infty$ since $A$ can be covered by finitely many left translates of $W$. Finally, $\mu_A(AQ)<\infty$ since $Q\seq P\seq W\seq AA\inv$ and $A$ has finite tripling. Altogether,
\[
\mu_Q(AQ)=\mu_Q(P)\mu_P(W)\mu_W(A)\mu_A(AQ)<\infty.
\]

Now, by Lemma \ref{lem:RCL}, there is a finite set $C\seq A$ such that $A\seq CQ^2\seq CP$, and a finite set $F\seq A\backslash Z$ such that $A\backslash Z\seq FQ^2\seq FP$. 
Without loss of generality, assume $F\seq C$. Define $D=\{g\in F:\mu_A(gP\backslash A)=0\}$. Note that, since $F\cap Z=\emptyset$, we have $\mu_A(gP\cap A)=0$ for all $g\in E:=F\backslash D$. So $\mu_A(DP\backslash A)=0$ and $\mu_A(EP\cap A)=0$. Therefore,  to show $\mu_A((A\smd DP)\backslash Z)=0$, it suffices to show $A\backslash DP\seq Z\cup (EP\cap A)$. So fix $x\in A\backslash DP$, and suppose $x\not\in Z$. Then $x\in A\backslash Z\seq FP$, and so there is some $g\in F$ such that $x\in gP$. Since $x\not\in DP$, it follows that $g\in E$, and so $x\in EP\cap A$.
\end{proof}

\begin{remark}
In the previous proof, by employing \cite{BGT} to trade the approximate group $W$ for the coset nilprogression $P$, we lose ``definability" of $P$  in terms of $\cR_{\!A}$. However, we do maintain $Z\in\cR_{\!A}$. Therefore, in the statement of Theorem \ref{thm:NIPgen}, one could add that the error set $Z$ is a Boolean combination of bi-translates of $A$ (of ``bounded complexity" as detailed in Remark \ref{rem:complex}).
\end{remark}

\subsection{Proof of Theorem \ref{thm:NIPgen}}\label{sec:NIPgenpf}
We will postpone discussion of the abelian case to the end of the proof. So suppose the main statement of Theorem \ref{thm:NIPgen} fails. Then we have fixed $d,k\geq 1$ and  $\epsilon>0$ such that, for any integer $i\geq 1$, there is a group $G_i$ and a finite $d$-NIP set $A_i\seq G_i$ with $k$-tripling such that, if $P\seq G_i$ is a coset nilprogression of rank and step $i$, and in $i$-normal form, and $Z\seq A_iP$ satisfies $|Z|<\epsilon|A_i|$, then it is not the case that all three of the following properties hold:
\begin{enumerate}[$(i)$]
\item $P\seq A_iA_i\inv\cap A_i\inv A_i$ and $A_i\seq CP$ for some $C\seq A_i$ with $|C|\leq i$.
\item There is a set $D\seq C$ such that $|(A_i\smd DP)\backslash Z|<\epsilon|P|$.
\item For any $g\in G_i\backslash Z$, either $|gP\cap A_i|<\epsilon|P|$ or $|gP\backslash A_i|<\epsilon|P|$.
\end{enumerate}

Let $\cU$ be a nonprincipal ultrafilter on $\Z^+$, and let $G=\prod_{\cU}G_i$ and $A=\prod_{\cU}A_i\seq G$. Then $A$ is pseudofinite and $d$-NIP, and we have $\mu_A(A^3)\leq k<\infty$. By Theorem \ref{thm:mainUP}, there is an internal coset nilprogression $P\seq G$ in normal form and an internal set $Z\seq G$, with $\mu_A(Z)<\epsilon$, satisfying the following properties:
\begin{enumerate}[\hspace{5pt}$\ast$]
\item $P\seq AA\inv\cap A\inv A$ and $A\seq CP$ for some finite $C\seq A$.
\item There is a set $D\seq C$ such that $\mu_A((A\smd DP)\backslash Z)=0$.
\item  If $g\in G\backslash Z$ then $\mu_A(gP\cap A)=0$ or $\mu_A(gP\backslash A)=0$.
\end{enumerate}

Write $Z=\prod_{\cU}Z_i$ where each $Z_i$ is a subset of $G_i$. By definition, there are $c,r,s\in\N$ such that $P=\prod_{\cU}P_i$, where each $P_i\seq G_i$ is a coset nilprogression of rank $r$ and step $s$ in $c$-normal form. Set $n=|C|$. Let $I\seq\Z^+$ be the set of $i\in\Z^+$ such that the following properties hold:
\begin{enumerate}[\hspace{5pt}$(1)$]
\item $P_i\seq A_iA_i\inv\cap A_i\inv A_i$ and $A_i\seq C_iP_i$ for some $C_i\seq A_i$ with $|C_i|\leq n$.
\item There is a set $D_i\seq C_i$ such that $|(A_i\smd D_iP_i)\backslash Z_i|/|A_i|<\epsilon/n$.
\item  If $g\in G_i\backslash Z_i$ then $|gP_i\cap A_i|/|A_i|<\epsilon/n$ or $|gP_i\backslash A_i|/|A_i|<\epsilon/n$.
\item $Z_i\seq A_iP_i$ and $|Z_i|/|A_i|<\epsilon$.
\end{enumerate}
Then $I\in\cU$ and so, since $\cU$ is nonprincipal, we may choose some $i\in I$ such that $i\geq\max\{c,n,r,s\}$. Note that $|A_i|\leq n|P_i|$ by condition $(1)$, and so $\epsilon/n\leq\epsilon|P_i|/|A_i|$. Altogether, conditions $(1)$ through $(4)$ contradict the choice of $G_i$ and $A_i$.

Finally, we explain how to get proper coset progressions when restricting to abelian groups. From the above proof, it is clear that we only need to ensure that, in the statement of Theorem \ref{thm:mainUP}, if $G=\prod_{\cU}G_i$ is abelian then $P$ can be assumed to be an ultraproduct of proper coset progressions (i.e., we may take $c=1$ in the proof). For this, one can transfer Theorem \ref{thm:GR} to the ultraproduct $G$, and use this in place of Theorem \ref{thm:BGTpf}. Alternatively, it is easy to check that if $G$ is abelian and $P\seq G$ is an internal coset progression in normal form, then there is an internal proper coset progression $P'\seq P$ such that $\mu_{P'}(P)<\infty$. So Theorem \ref{thm:BGTpf} also recovers proper coset progressions in abelian groups.
\qed

\subsection{Finite exponent: pseudofinite version}

The goal of this subsection is to prove a pseudofinite analogue of Theorem \ref{thm:NIPexp}, which is our main result for NIP sets of small tripling in the setting of bounded exponent. The key use of bounded exponent is essentially the fact that a locally compact Hausdorff torsion group is totally disconnected, which is a consequence of the Gleason-Yamabe theorem. In the model-theoretic setting, this fact can be used to show that if $G$ is a sufficiently saturated group and $\Sigma$ is an ind-definable subgroup of finite exponent, then any type-definable bounded-index normal subgroup of $\Sigma$ is in fact an intersection of definable \emph{subgroups} of $\Sigma$.  This is a key ingredient in the proof of Theorem \ref{thm:BGTbe}  in \cite{BGT} and \cite{HruAG}. Our version follows \cite[Corollary 5.6]{vdDag}, which relaxes the finite exponent assumption somewhat. We also add extra control on the definability properties of the objects involved.

\begin{proposition}\label{prop:GYbd}
Let $G^*$ be a sufficiently saturated group and let $\cR$ be a subring of $\Def(G^*)$. Suppose $\Sigma$ is an $\cR$-ind-definable subgroup of $G^*$, $\Gamma$ is an $\cR$-type-definable bounded-index normal subgroup of $G$, and $\cR$ is left-$\Sigma$-invariant. If $X\seq \Sigma$ is a definable set of finite exponent containing $\Gamma$, then there is an $\cR$-definable subgroup $H\leq \Sigma$ such that $\Gamma\leq H\seq X$.
\end{proposition} 
\begin{proof}
Let $\pi\colon \Sigma\to \Sigma/\Gamma$ be the quotient map. By Fact \ref{fact:LCLT}, there is an open identity neighborhood $U\seq\Sigma/\Gamma$ such that $\pi\inv(U)\seq X$. By the Gleason-Yamabe-Theorem (see, e.g., \cite[Theorem B.17]{BGT}), there is an open subgroup $K\leq\Sigma/\Gamma$, and a compact normal subgroup $N\leq K$ such that $N\seq U$ and $K/N$ is a connected Lie group. Since $X$ has finite exponent in $\Sigma$, it follows that $K/N$ is a connected Lie group with an identity neighborhood  of finite exponent. 
Therefore $K/N$ is trivial (e.g., by considering $1$-parameter subgroups).
So $K$ is a compact-open subgroup of $\Sigma/\Gamma$ contained in $U$. Let $H=\pi\inv(K)$. Then $H\in\cR$ by Theorem \ref{thm:LCLT} (and Remark \ref{rem:LCLT}). Moreover, $H$ is a subgroup of $\Sigma$ and $\Gamma\seq H\seq \pi\inv(U)\seq X$. 
\end{proof}

We can now prove a pseudofinite version of Theorem \ref{thm:NIPexp}. 

\begin{theorem}\label{thm:expUP}
Let $G=\prod_{\cU}G_i$ where each $G_i$ is a group and $\cU$ is an ultrafilter on some index set $I$.
Suppose $A\seq G$ is internal, NIP, and pseudofinite with finite tripling, and assume that $AA\inv\cap A\inv A$ has finite exponent. Then for any $\epsilon>0$, there is an $\cR_{\!A}$-definable subgroup $H$ of $G$ and a set $Z\seq AH$, which is a union of left cosets of $H$ with $\mu_A(Z)<\epsilon$, satisfying the following properties.
\begin{enumerate}[$(i)$]
\item $H\seq AA\inv\cap A\inv A$ and $A\seq CH$ for some finite $C\seq A$.
\item There is a set $D\seq C$ such that $\mu_A((A\backslash Z)\smd DH)=0$.
\item  If $g\in G\backslash Z$ then $\mu_A(gH\cap A)=0$ or $\mu_A(gH\backslash A)=0$.
\end{enumerate}
\end{theorem}
\begin{proof}
We start as in the proof of Lemma \ref{lem:UP1} and move to a saturated extension $G^*$. Let $A_*=A(G^*)$. By Proposition \ref{prop:GYbd}, there is a decreasing sequence $(H_n)_{n=0}^\infty$ of $\cR_{\!A_*}$-definable subgroups of $\langle A_*\rangle$ such that $\Gamma_{\!A_*}=\bigcap_{n=0}^\infty H_n$. Since $\Gamma_{\!A^*}$ is countably $\cR_{\!A_*}$-type-definable over $G$, we may assume that each $H_n$ is $\cR_{\!A_*}$-definable over $G$. 
Now fix $\epsilon>0$. By Corollary \ref{cor:oneY}, there is a definable set $Z_*\seq G^*$ and some $n\geq 0$ such that  $\mu_A(Z_*)<\epsilon$, and if $g\in G^*\backslash Z_*$ then $\mu_A(gH_n\cap A_*)=0$ or $\mu_A(gH_n\backslash A_*)=0$. Without loss of generality, we may assume that $Z_*$ is a union of left cosets of $H_n$.

Let $Z=Z_*(G)$ and $H=H_n(G)$. Then, working in $G$, we have  $H\seq AA\inv\cap A\inv A$, $A\seq CH$ for some finite $C\seq G$, $\mu_A(Z)<\epsilon$, and if $g\in G\backslash Z$ then $\mu_A(gH\cap A)=0$ or $\mu(gH\backslash A)=0$.  We may assume $Z\seq AH$ since if $g\not\in AH$ then $gH\cap A=\emptyset$. Similarly, we may assume that $gH\cap A\neq\emptyset$ for all $g\in C$ and thus, after changing coset representatives, assume $C\seq A$. Now let $D=\{g\in C\backslash Z:\mu_A(gH\backslash A)=0\}$ and $E=\{g\in C\backslash Z:\mu_A(gH\cap A)=0\}$. By construction, 
\[
(A\backslash Z)\smd DH\seq Z\cup (DH\backslash A)\cup (EH\cap A),
\]
and so $\mu_A((A\backslash Z)\smd DH)=0$. 
\end{proof}

\subsection{Proof of Theorem \ref{thm:NIPexp}}
The proof involves a similar ultraproduct construction as in the proof of Theorem \ref{thm:NIPgen} in Section \ref{sec:NIPgenpf}. Thanks to the bounded exponent assumption, one can use Theorem \ref{thm:expUP} in place of Theorem \ref{thm:mainUP} to replace coset nilprogressions by subgroups. Details are left to the reader. \qed

\begin{remark}\label{rem:complex}
In the proof of Theorem \ref{thm:NIPexp}, one can also control the ``complexity" of $H$ as a Boolean combination of bi-translates of $A$ in the following way. Given a Boolean function $f(x_1,\ldots,x_n)$ in finitely many variables, a set $A\seq G$, and $n$-tuples $\gbar,\hbar$ of elements in $G$, let $f^A(\gbar,\hbar)$ be the resulting Boolean combination of $g_1Ah_1,\ldots,g_nA h_n$. Now let $(f_t)_{t=0}^\infty$ be an enumeration of all such Boolean functions. Then, in the conclusion, of Theorem \ref{thm:NIPexp}, we can say more precisely that there are $m,t\leq O_{d,k,r,\epsilon}(1)$ and $n$-tuples $\gbar,\hbar$ of elements in $A^{\pm m}$, where $f_t$ is in $n$ variables, such that $H=f^A_t(\gbar,\hbar)$. In fact, since we work with the ring $\cR_{\!A}$, we may restrict to Boolean functions in disjunctive normal form, in which each disjunct contains at least one positive literal. 
\end{remark}

\subsection{Theorem \ref{thm:NIPexp} via NIP arithmetic regularity in finite groups}

As explained at the start of Section \ref{sec:dom}, Theorem \ref{thm:GCD} was obtained as a generalization of a related statement for pseudofinite groups from \cite{CPpfNIP}. This result from \cite{CPpfNIP} is used in \cite[Theorem 3.1]{CPTNIP} to show that if $G$ is a pseudofinite group of finite exponent, and $A\seq G$ is internal and NIP, then $A$ can be approximated by (cosets of) a definable finite-index normal subgroup of $G$. Now, in the proof of Theorem \ref{thm:expUP}, we obtained an internal \emph{pseudofinite group} $H$ contained in $AA\inv\cap A\inv A$, such that $A\seq CH$ for some finite set $C$. Therefore, one can apply the results from \cite{CPTNIP} to directly to $H$, in order to approximate any set of the form $aH\cap A$ by a finite-index subgroup $H_a$ of $H$. This leads to an approximation of $A$ by $\bigcap_{a\in C}H_a$. We leave the details (which involve some careful calculation) to the interested reader. 

Altogether, this yields a proof of Theorem \ref{thm:NIPexp} that only requires Theorem \ref{thm:Stabs} and Proposition \ref{prop:GYbd} (in addition to \cite{CPpfNIP,CPTNIP}). From a broad perspective, this strategy is analogous to the proof of the \emph{stable} Freiman-Ruzsa result of from \cite{MPW}, which we discuss in the next section (see also Remark \ref{rem:MPW}).

\subsection{The stable case}\label{sec:stable}

In this section, we  discuss the \emph{stable} Freiman-Rusza result of Martin-Pizarro, Palac\'{i}n, and Wolf \cite{MPW}.

\begin{definition}
Let $G$ be a group. Then a set $A\seq G$ is \textbf{$d$-stable} if there do not exist $a_1,\ldots,a_d,b_1,\ldots,b_d\in G$ such that $a_ib_j\in A$ if and only if $i\leq j$.
\end{definition}

\begin{remark}\label{rem:stable}
Note that stability of $A\seq G$ can be rephrased in terms of forbidden subgraphs in the bipartite graph $\Gamma_G(A)=(G,G;yx\in A)$. Indeed, $A$ is $d$-stable if and only if $\Gamma_G(A)$ omits $([d],[d];\leq)$. One can easily check that a $d$-stable subset of a group is $d$-NIP. Recall also that, as stated at the start of Section \ref{sec:intro4}, a nonempty set $A\seq G$ is a coset of subgroup of $G$ if and only if $A$ is $2$-stable. 
\end{remark}

As with NIP, we call a  set $A\seq G$ \textbf{stable} if it is $d$-stable for some $d\geq 1$. More generally, a formula $\phi(\xbar;\ybar)$ is \textbf{stable} with respect to a first-order structure $M$ if, for some $d\geq 1$, there do not exist $\abar_1,\ldots,\abar_d\in M^{\xbar}$ and $\bbar_1,\ldots,\bbar_d\in M^{\ybar}$ such that $M\models\phi(\abar_i,\bbar_j)$ if and only if $i\leq j$. One of the most important properties of  stable formulas is the following result of Shelah (see Theorem II.2.2 in \cite{Shbook}). 

\begin{theorem}[Shelah \cite{Shbook}]\label{thm:DOT}
Let $M$ be a first-order structure. Suppose $\phi(\xbar;\ybar)$ is a stable formula and $p\in S_\phi(M)$ is a complete $\phi$-type over $M$. Then the set $\{\bbar\in M^{\ybar}:\phi(\xbar;\bbar)\in p\}$ is definable.
\end{theorem}

By Remark \ref{rem:stable}, one could approximate finite stable sets of small tripling by coset nilprogressions as in Theorem \ref{thm:NIPgen}. However, stable sets are much more well-behaved for two reasons. First, in the setting of Theorem \ref{thm:Stabs}, if $A$ is stable then $\Gamma_{\!A}$ is an intersection of definable \emph{subgroups} of $\langle A \rangle$. Second, in the setting of Theorem \ref{thm:GCD}, if $A$ is stable then $\E_X=\emptyset$ for any $X\in\cR_{\!A}$. Both of these statements rely on Theorem \ref{thm:DOT}, and will be shown in the proof of Theorem \ref{thm:stabUP} below.

Let us now state the stable Freiman-Ruzsa result, which is nearly identical to \cite[Theorem A]{MPW}, but with some small additions.

\begin{theorem}[Martin-Pizarro, Palac\'{i}n, Wolf]\label{thm:MPW}
Suppose $G$ is a group and $A\seq G$ is a finite $d$-stable set with $k$-tripling. Given $\epsilon>0$, there is a subgroup $H\leq G$ satisfying the following properties.
\begin{enumerate}[$(i)$]
\item $H\seq AA\inv\cap A\inv A$ and $A\seq CH$ for some $C\seq A$ with $|C|\leq O_{d,k,\epsilon}(1)$. 
\item There is a set $D\seq C$ such that $|A\smd DH|<\epsilon|H|$.
\item For any $g\in G$, either $|gH\cap A|<\epsilon|H|$ or $|gH\cap A|>(1-\epsilon)|H|$.
\end{enumerate}
Moreover, $H$ is a finite Boolean combination of bi-translates of $A$.
\end{theorem}

We will give a proof of the previous theorem, which only requires Lemma \ref{lem:GCD}$(d)$, Theorem \ref{thm:DOT}, and the features of $\Gamma_{\!A}$ proved in Section \ref{sec:NIPLie}. See Remark \ref{rem:MPW} for a discussion of how this compares to the proof in  \cite{MPW}.  As was the case with Theorems \ref{thm:NIPgen} and \ref{thm:NIPexp}, it suffices to prove the following pseudofinite version. 

\begin{theorem}\label{thm:stabUP}
Let $G=\prod_{\cU}G_i$, where each $G_i$ is a group and $\cU$ is a nonprincipal ultrafilter on some index set $I$. Suppose $A\seq G$ is internal, stable, and pseudofinite with finite tripling. Then there is an $\cR_{\!A}$-definable subgroup $H$ of $G$ satisfying the following properties.
\begin{enumerate}[$(i)$]
\item $H\seq AA\inv\cap A\inv A$ and $A\seq CH$ for some finite $C\seq A$.
\item There is a set $D\seq C$ such that $\mu_A(A\smd DH)=0$.
\item If $g\in G$ then $\mu_A(gH\cap A)=0$ or $\mu(gH\backslash A)=0$.
\end{enumerate}
\end{theorem}
\begin{proof}
Note that $(ii)$ and $(iii)$ are equivalent. We prove $(i)$ and $(iii)$. Let $G^*\succ G$ be sufficiently saturated and set $A_*=A(G^*)$.  Let $\cB$, $\cB^\ell$, and $\cB^r$ denote the Boolean algebras of subsets of $G^*$ generated by all bi-translates, left translates, and right translates of $A_*$, respectively. Let $\Gamma=\Stab^\ell_{\!A}(A_*)$. 

\medskip

\noindent\textit{Claim 1:} $\Gamma$ is $\cR^r_{\!A_*}$-definable over $G$. 

\noindent\textit{Proof:} Recall that $\Gamma$ is $\cR^r_{\!A_*}$-type-definable over $G$ by Lemma \ref{lem:NIPLie}. Fix a $\mu_A$-wide type $p\in S(\cB^r)$ such that $p\models \Gamma$. Then $p$ is a complete $\phi$-type over $G$ for the stable formula $\phi(x;y):=A(x\cdot y)$, and so $H:=\Stab^r_{G^*}(p)$ is definable by Theorem \ref{thm:DOT}. Note that $H\leq \Gamma$ since $p\models\Gamma$. Note also that $H=\Stab^r_{\!A^*}(p{\upharpoonright}\cR^r_{\!A_*})$, and so $\Gamma_{\!A_*}\leq H$ by Lemma \ref{lem:Stabs}$(a)$. Altogether, $H$ is a definable subgroup of the type-definable group $\Gamma$, and $[\Gamma:H]$ is bounded since $\Gamma_{\!A_*}\leq H$. By a straightforward compactness exercise, $H$ has finite index in $\Gamma$ (e.g., set $X=\Gamma$ and $Y=H$ in \cite[Lemma 1.6]{HruAG}). So $\Gamma$ is definable, and therefore $\cR^r_{\!A_*}$-definable over $G$ by saturation.\claim
\medskip

By Claim 1 and Remark \ref{rem:Stabs}, $\Gamma_{\!A_*}=\bigcap_{n=0}^\infty H_n$, where each $H_n\leq G^*$ is $\cR_{\!A_*}$-definable over $G$ and contained in $A_*A_*\inv\cap A_*\inv A_*$. Note also that, for all $n\geq 0$, $A_*$ is covered by finitely many left cosets of $H_n$.

\medskip

\noindent\emph{Claim 2:} $\E_{A_*}=\emptyset$.

\noindent\emph{Proof:}  Let $p_*\in S(\cB)$ be a $\mu_A$-wide type such that $p_*\models\Gamma_{\!A}$, and let $p=p_*{\upharpoonright}\cR_{\!A_*}$. Then $\E_{A_*}\seq\partial\B^p_{\!A_*}$ by Lemma \ref{lem:GCD}$(d)$, and so it suffices to show that $\B^p_{\!A_*}$ is clopen in $\langle A_*\rangle/\Gamma_{\!A_*}$. Note that $\pi\inv(\B^p_{A_*})=\{a\in G^*:A_*\in ap\}$. Moreover, if $p_0=p_*{\upharpoonright}\cB^\ell$, then $p_0$ is a complete $\phi$-type with respect to the stable formula $\phi(x;y):=A(y\cdot x)$, and $\phi(x;a)\in p_0$ if and only if $A_*\in ap$. So $\pi\inv(\B^p_{A_*})$ is definable by Theorem \ref{thm:DOT}, and thus $\B^p_{\!A_*}$ is compact-open by Fact \ref{fact:LCLT}$(b)$.\claim
\medskip

By Claim 2, we can carry out the proof of Corollary \ref{cor:oneY} with $Z=\emptyset$. So there is some $n\geq 0$ such that if $g\in G^*$ then $\mu_A(gH_n\cap A_*)=0$ or $\mu_A(H_n\backslash A_*)=0$. This yields the desired result in $G$ as usual.
\end{proof}

In the context of the previous proof, one can similarly show that $\Stab^r_{\!A}(A_*)$ is in $\cR^\ell_{\!A_*}$, and that $\E_X=\emptyset$ for any $X\in\cR_{\!A_*}$. Note that we also obtain an explicit description of the subgroup $H$ in the statement of Theorem \ref{thm:stabUP}, namely, $H=\bigcap_{g\in F}g\Stab^\ell_{\!A}(A)g\inv$ for some finite $F\subset G$.

\begin{remark}\label{rem:MPW}
The proof of Theorem \ref{thm:MPW} in \cite{MPW} follows the same strategy with ultraproducts, and a similar argument involving definability of types is employed to find a definable subgroup $H\seq AA\inv$ such that $A$ is covered by finitely many left cosets of $H$. They then apply the stable arithmetic regularity lemma of the authors and Terry from \cite{CPT} to obtain the remaining properties. We also note that a quantitative version of Theorem \ref{thm:MPW} was proved recently by the first author \cite{CoQSAR}.
\end{remark}

\section{The abelian case}\label{sec:abelian}
 
The goal of this section is to prove Freiman-Ruzsa results for finite NIP sets with small doubling in abelian groups \emph{directly} from arithmetic regularity results for dense NIP sets in finite groups. The reason for doing this is that these regularity results for finite abelian groups are often effective (although sometimes qualitatively weaker; see Remark \ref{rem:qual} below). For example, using work of Alon, Fox, and Zhao \cite{AFZ}, we will obtain an effective Freiman-Ruzsa result for finite NIP sets with small doubling in abelian groups of bounded exponent. Altogether, the material in this section does not require the theorems proved above  (although Theorem \ref{thm:NIPgen-ab} crucially relies on \cite{CPpfNIP} and \cite{CPTNIP}).

Let us now state the two main results that will be proved in this section. The first is simply a restatement of Theorem \ref{thm:NIPgen} for abelian groups.

\begin{theorem}\label{thm:NIPgen-ab}
Suppose $G$ is an abelian group and $A\seq G$ is a finite $d$-NIP set with $k$-doubling. Given $\epsilon>0$, there is a proper coset progression $P\seq G$ of rank $O_{d,k,\epsilon}(1)$, and a set $Z\seq A+P$ with $|Z|<\epsilon|A|$, satisfying the following properties.
\begin{enumerate}[$(i)$]
\item $P\seq A-A$ and $A\seq C+P$ for some $C\seq A$ with $|C|\leq O_{d,k,\epsilon}(1)$. 
\item There is a set $D\seq C$ such that $|(A\smd(D+P))\backslash Z|<\epsilon|P|$.
\item For any $g\in G\backslash Z$, either $|(g+P)\cap A|<\epsilon|P|$ or $|(g+P)\backslash A|<\epsilon|P|$.
\end{enumerate}
\end{theorem}

We still do not obtain an effective bound on $O_{d,k,\epsilon}(1)$. However, the only barrier to obtaining such bounds would be an effective version of Theorem \ref{thm:CPTNIPab} below (see Remark \ref{rem:abSis}). On the other hand, using work of Alon, Fox, and Zhao \cite{AFZ}, we will obtain an effective result in the setting of bounded exponent. The statement uses the following notation (which is explained by Lemma \ref{lem:GR}). Given integers $k,r\geq 1$ define
\[
c_r(k)=
\begin{cases}
(r-1) k^{12} & \text{if $r$ is prime,}\\
k^{2}r^{2k^2-2} & \text{otherwise.}
\end{cases}
\]
We now state a quantitative version of Theorem \ref{thm:NIPexp}  for abelian groups.

\begin{theorem}\label{thm:NIPexp-ab}
Suppose $G$ is an abelian group of exponent $r$, and $A\seq G$ is a finite $d$-NIP set with $k$-doubling. Fix $0<\delta\leq 1$ and let $\epsilon=\delta/c_r(k)$. Then there is a subgroup $H\leq G$ satisfying the following properties.
\begin{enumerate}[$(i)$]
\item $H\seq A-A$ and $A\seq C+H$ for some $C\seq A$ with $|C|\leq \exp(O_r(d^8))\epsilon^{\nv d}$.
\item There is a set $D\seq C$ such that $|A\smd (D+H)|<\delta|A|$.
\item There is a set $Z\seq A+H$, which is a union of cosets of $H$ with $|Z|<\delta^{1/2}|A|$, such that if $g\in G\backslash Z$ then $|(g+H)\cap A|<\epsilon^{1/4}|H|$ or $|(g+H)\backslash A|<\epsilon^{1/4}|H|$.
\end{enumerate}
\end{theorem}

The implied constant $O_r(1)$ in the bound on $|C|$ from Theorem \ref{thm:NIPexp-ab} is the same as in \cite[Theorem 11.1]{SanBR} up to an absolute multiplicative factor.

\begin{remark}\label{rem:qual}
Although Theorem \ref{thm:NIPexp-ab} is quantitative, the structural approximation of the set $A$ in condition $(ii)$ is qualitatively weaker than in the non-abelian counterparts proved using ultraproducts. Indeed, Theorem \ref{thm:NIPexp}$(ii)$ implies  Theorem \ref{thm:NIPexp-ab}$(ii)$ modulo replacing $\delta$ with $\delta/k^2$ (as explained after Theorem \ref{thm:NIPgen}). On the other hand, Theorem \ref{thm:NIPexp}$(ii)$ provides a finer description of the error between $A$ and $D+H$, which cannot be obtained directly from Theorem \ref{thm:NIPexp-ab}$(ii)$ alone.
\end{remark}

Theorems \ref{thm:NIPgen-ab} and \ref{thm:NIPexp-ab} will both be proved using arithmetic regularity results for subsets of \emph{finite groups}. The key ingredient that makes this reduction possible is Green and Ruzsa's ``modeling lemma", which  lies at the heart of their proof of Freiman's Theorem for arbitrary abelian groups (Theorem \ref{thm:GR} above). Roughly speaking, this lemma says that a finite set of small doubling in an abelian group can be algebraically modeled by a dense set in a finite group. In order to state the result precisely, we first recall the definition of \emph{Freiman isomorphism}.

\begin{definition}
Suppose $G$ and $G'$ are abelian groups, and fix $A\seq G$, $A'\seq G'$, and $s\geq 1$. A function $\phi:A\to A'$ is a \textbf{Freiman $s$-isomorphism} if $\phi$ is bijective and for all $a_1,\ldots,a_s,b_1\ldots,b_s\in A$,
\[
a_1+\ldots+a_s=b_1+\ldots+b_s\miff \phi(a_1)+\ldots+\phi(a_s)=\phi(b_1)+\ldots+\phi(b_s).
\]
\end{definition}

It is easy to check that a Freiman $s$-isomorphism is a Freiman $s'$-isomorphism for any $1\leq s'\leq s$. Next we state the ``modeling lemma". 

\begin{lemma}[Green \& Ruzsa \cite{GrRuz2,GrRuz}]\label{lem:GR}
Suppose $G$ is an abelian group and $A\seq G$ is a finite set with $k$-doubling. Then, for any $s\geq 1$, there is an abelian group $G'$, with $|G'|\leq (10sk)^{10k^2}|A|$ and a Freiman $s$-isomorphism $\phi\colon A\to A'$ for some $A'\seq G'$ with $0\in A'$. Moreover, if $G$ has exponent $r$ then one may  assume that $G'$ has exponent $r$ and size at most $k^2r^{2k^2-2}|A|$. If $r$ is prime, then one may further assume $|G'|\leq (r-1)k^{2s}|A|$.
\end{lemma}
\begin{proof}
The main claim is \cite[Proposition 1.2]{GrRuz} (one can always ensure $0\in A'$ by translating the isomorphism). If $G$ has exponent $r$ then, by \cite[Theorem 6.1]{GrRuz2}, there is $G'\leq G$ of size $k^2r^{2k^2-2}|A|$, and some $x\in G$, such that $A\seq x+G'$. So $\phi\colon A\to A-x$ such that $\phi(a)=a-x$ is a Freiman $s$-isomorphism for all $s\geq 1$. The final claim when $r$ is prime can be shown by directly generalizing \cite[Proposition 6.1]{GrRuz} (which deals with $r=s=2$). 
\end{proof}

The last ingredient we need is the fact that  NIP sets are preserved by Freiman isomorphism, which was first shown by Sisask \cite[Lemma 4.4]{SisNIP} (using a slightly different setup). Translated to our framework, Sisask's result implies that if $A\seq G$ is $d$-NIP in $G$ and Freiman $2$-isomorphic to $A'\seq G'$, then $A'$ is at worst $(d+1)$-NIP in $G'$. (This discrepancy is necessary in general, e.g., consider $A=A'=G$ and let $G'$ be a group properly containing $G$.) Although this is only a minor loss, the bounds in \cite{AFZ} are quite sharp in terms of $d$. Thus it is worth noticing that that the loss can be avoided when applying the modeling lemma.

 Given a group $G$, a set $A\seq G$, and $\cS\seq\cP([d])$ for some $d\geq 1$, we say that $A$ \emph{cuts out $\cS$ in $G$} if there are $x_1,\ldots,x_d\in G$ and $g_S\in G$ for $S\in\cS$ such that, given $i\in[d]$ and $S\in\cS$, $x_i\in g_S+A$ if and only if $i\in S$. Note that $A$ is $d$-NIP in $G$ if and only if it does not cut out $\cP([d])$ in $G$.

\begin{proposition}\label{prop:FItame}
In the statement of Lemma \ref{lem:GR}, if $s\geq 4$ and $A$ is $d$-NIP  in $G$ then one may assume $A'$ is $d$-NIP in $G'$. 
\end{proposition}
\begin{proof}
Assume $A$ is $d$-NIP in $G$. Note that any translate of $A$ is $d$-NIP and Freiman $s$-isomorphic to $A$ for any $s$. So we may assume $0\in A$. Let $\phi\colon A\to A'\seq G'$ be a Freiman $s$-isomorphism as in Lemma \ref{lem:GR}, with $s\geq 4$.   Without loss of generality, we may assume $G'$ is generated by $A'$. In this case, we show that $A'$ is $d$-NIP in $G'$. For a contradiction, suppose $A'$ cuts out $\cP([d])$ in $G'$, witnessed by $X'=\{x'_1,\ldots,x'_d\}\seq G'$ and $\{g_S:S\seq[d]\}\seq G'$. After translating by $\nv g_{[d]}$, we may assume $X'\seq A'$.

 Let $\cS=\cP([d])\backslash\{\emptyset\}$. Then, for any $S\in\cS$, we have $(g_S+A')\cap X'\neq\emptyset$, and so $g_S\in A'-A'$. Let $g_S=a_S-b_S$, where $a_S,b_S\in A'$, and set $h_S=\phi\inv(a_S)-\phi\inv(b_S)\in G'$. We claim that, for any $S\in\cS$, we have $\phi\inv(x'_i)\in h_S+A$ if and only if $i\in S$, and so $A$ cuts out $\cS$ in $G$. So fix $S\in\cS$ and $i\in[d]$. Suppose $\phi\inv(x'_i)=h_S+a$ for some $a\in A$, and let $a'=\phi(a)$. Then $\phi\inv(x'_i)+\phi\inv(b_S)=\phi\inv(g_S)+\phi\inv(a')$, and so $x'_i=g_S+a'\in g_S+A'$, which implies $i\in S$. Conversely, if $i\in S$, then there is some $a'\in A'$ such that $x'_i+b_S=a_S+a'$, and so $\phi\inv(x'_i)=h_S+\phi\inv(a')\in h_S+A$. 

Now, since $A$ does not cut out $\cP([d])$, it follows that $(h+A)\cap X\neq \emptyset$ for all $h\in G$, i.e., $G=X-A$. Define $\psi\colon A-A\to A'-A'$ such that, given $a,b\in A$, $\psi(a-b)=\phi(a)-\phi(b)$. Then $\psi$ is a well-defined Freiman $\lfloor s/2\rfloor$-isomorphism with domain $G$. Moreover, $\psi(G)$ is a subgroup of $G'$ containing $A'=\psi(A)$ (recall that $A$ contains $0$). Therefore $G'=\psi(G)=\psi(X-A)=\phi(X)-\phi(A)=X'-A'$. But this is a contradiction, since $g_\emptyset\not\in X'-A'$. 
\end{proof}

The proofs of Theorems \ref{thm:NIPgen-ab} and \ref{thm:NIPexp-ab} will follow a common strategy. Given an abelian group $G$ and a finite NIP set $A\seq G$ with small doubling, we will use Lemma \ref{lem:GR} (and Proposition \ref{prop:FItame}) to move to a finite abelian group and an NIP dense set $A'\seq G'$. We will then apply an appropriate arithmetic regularity result in $G'$ to show that $A'$ can be approximated by algebraically well-structured objects that are preserved by Freiman isomorphism, and thus can be pulled back to approximate the original set $A$. In Theorem \ref{thm:NIPexp-ab}, these algebraic objects will be subgroups, and will transfer easily. However, for Theorem \ref{thm:NIPgen-ab}, we will start with Bohr sets (defined below) in $G'$, which themselves will need to be approximated by coset progressions before being pulled back to $G$. In order to do this secondary approximation, we will need a qualitatively sharper arithmetic regularity result (Theorem \ref{thm:CPTNIPab}) than what is available in the literature of quantitative tools in additive combinatorics. It is for this reason that Theorem \ref{thm:NIPgen-ab} remains ineffective. In order to state Theorem \ref{thm:CPTNIPab}, we need the following definition. 

\begin{definition}
Let $G$ be a group, and suppose $\tau\colon G\to (\R/\Z)^r$ is a group homomorphism, where we view $(\R/\Z)^r$ as an additive group with identity $\boldsymbol{0}$. Given $\delta>0$, define
\[
B_\tau(\delta)=\{x\in G:d(\tau(x),\boldsymbol{0})<\delta\},
\]
where $d$ is the $r$-fold product of the arclength metric on $\R/\Z=S^1$. A subset of $G$ obtained in this way is called a \textbf{$(\delta,r)$-Bohr set in $G$}.
\end{definition}

\begin{theorem}[C., P., Terry]\label{thm:CPTNIPab} Given $d\in\Z^+$, $\alpha,\epsilon\in\R^+$, and $f\colon (0,1]\times\N\to \R$, there is an integer $n=n(d,\alpha,\epsilon,f)$ such that the following holds. Suppose $G$ is a finite abelian group and $A\seq G$ is $d$-NIP, with $|A|\geq\alpha|G|$. Then there is a $(\delta,r)$-Bohr set $B\seq G$ and a set $Z\seq G$ such that:
\begin{enumerate}[$(i)$]
\item $\delta\inv, r\leq n$, $|Z|<\epsilon|G|$, and $B\seq A-A$,
\item for any $g\in G\backslash Z$, either $|(g+B)\cap A|<f(\delta,r)|B|$ or $|(g+B)\backslash A|<f(\delta,r)|B|$.
\end{enumerate}
\end{theorem}
\begin{proof}
This is essentially the abelian case of \cite[Lemma 5.6]{CPTNIP}, except we have added an extra assumption that $|A|\geq\alpha|G|$ and the extra conclusion $B\seq A-A$. To obtain this modification, note that in the proof of \cite[Lemma 5.6]{CPTNIP} if we have $G=\prod_{\cU}G_t$ and $A=\prod_{\cU}A_t$ with $|A_t|\geq\alpha|G_t|$ for all $t>0$, then we have the extra property that $\mu(A)\geq\alpha>0$, where $\mu$ is the $G$-normalized pseudofinite counting measure.  Let $G^*\succ G$ be sufficiently saturated. Following the proof of \cite[Lemma 5.6]{CPTNIP}, we obtain a definable approximate Bohr neighborhood $Y$ whose translates are regular for $A_*$ (outside a set of measure less than $\epsilon$). From the rest of the proof, one can see that it suffices to show $Y\seq A_*-A_*$. At this point, we turn to the proof of \cite[Theorem 5.5]{CPTNIP}, where $Y$ is obtained inside some definable set $W$ taken from a descending chain $(W_i)_{i=0}^\infty$ of definable sets whose intersection is $G^{00}_{\theta^r}$, where $\theta^r(x;y,z)$ is the formula $A(y\cdot x\cdot z)$. In this situation, we can can replace $W$ by $W_i$ for any large enough $i$, without affecting the rest of the proof. Therefore, it suffices to show $G^{00}_{\theta^r}\seq A_*-A_*$. This follows from the fact that there is a generic $\theta^r$-type $p$ containing $A_*$ by \cite[Proposition 3.12]{CPpfNIP}, and  $G^{00}_{\theta^r}=\Stab^\ell_G(p)$ by \cite[Theorem 3.12]{CPpfNIP}. 
\end{proof}

The next proposition collects several standard facts from additive combinatorics concerning coset progressions and Bohr sets.

\begin{proposition}\label{prop:TV}
Let $G$ be an abelian group.
\begin{enumerate}[$(a)$]
\item Suppose $P\seq G$ is a coset progression, and $\phi\colon P\to G'$ is a Freiman $2$-isomorphism from $P$ to a subset of an abelian group $G'$, with $\phi(0)=0$.  Then $\phi(P)$ is a coset progression of the same rank as $P$. Moreover, if $P$ is proper then so is $\phi(P)$.
\item Suppose $G$ is finite and $B=B_\tau(\delta)$ is a $(\delta,r)$-Bohr set in $G$.
\begin{enumerate}[$(i)$]
\item There is a proper coset progression $P\seq G$ of rank $s\leq r$ such that $B_\tau(s^{\nv 2s}\delta)\seq P\seq B_\tau(\delta)$. 
\item For any $A\seq G$ there is $X\seq A$ such that $|X|\leq (2/\delta)^r$ and $A\seq X+B$.
\end{enumerate}
\end{enumerate}
\end{proposition}
\begin{proof}
See \cite[Proposition 5.24]{TaoVu} for part $(a)$,\footnote{In \cite{TaoVu}, coset progressions are not assumed to be symmetric, but it is easy to check that a Freiman $2$-isomorphism preserving the identity also preserves symmetric sets.} 
and \cite[Lemma 4.22]{TaoVu} for part $(b)(i)$. Part $(b)(ii)$ follows from Lemmas 2.14\footnote{This is the abelian case of Ruzsa's Covering Lemma (Lemma \ref{lem:RCL}).} 
and 4.20 of \cite{TaoVu}, together with  the fact that $B_\tau(\delta/2)+B_\tau(\delta/2)\seq B_\tau(\delta)$. (See also \cite[Proposition 4.9]{CPTNIP}.)
\end{proof}

We can now prove the main theorems stated at the beginning of this section.

\begin{proof}[\textnormal{\textbf{Proof of Theorem \ref{thm:NIPgen-ab}}}]
Fix an abelian group $G$ and a finite $d$-NIP set $A\seq G$ with $k$-doubling.   By Lemma \ref{lem:GR} and Proposition \ref{prop:FItame}, there is an abelian group $G'$ of size at most $(100k)^{10k^2}|A|$, a $d$-NIP subset $A'\seq G'$ containing $0$, and a Freiman $10$-isomorphism $\phi\colon A\to A'$. Note that it is enough to prove the desired conditions for a translate of $A$, and so after shifting $\phi$ and $A$, we can further assume $0\in A$ and $\phi(0)=0$.

Fix $\epsilon>0$, and set $f(u,v)=\epsilon(u/2v^{2v})^{2v}$ and  $\alpha=(100k)^{\nv 10k^2}$. Let $n=n(d,\alpha, \epsilon\alpha,f)$ be as in Theorem \ref{thm:CPTNIPab}.
 Then, by Theorem \ref{thm:CPTNIPab}, there is a $(\delta,r)$-Bohr set $B=B_\tau(\delta)\seq G'$, with $\delta\inv ,r\leq n$, and a set $Z'\seq G'$, with $|Z|<\epsilon \alpha|G'|\leq \epsilon|A'|$, such that $B\seq A'-A'$ and, for all $x\in G'\backslash Z'$, either $|(x+B)\cap A'|<f(\delta,r)|B|$ or $|(x+B)\backslash A'|<f(\delta,r)|B|$. Set $\gamma=(\delta/2r^{2r})^r$, and note that $f(\delta,r)=\epsilon\gamma^2$. 

By Proposition \ref{prop:TV}$(b)(i)$, we may fix a proper coset progression $P'$ in $G'$ of rank $s\leq r$ such that $B':=B_\tau(s^{\nv 2s}\delta)\seq P'\seq B$. Note also that $|B|\leq \gamma|B'|$ by Proposition \ref{prop:TV}$(b)(ii)$, and so $f(\delta,r)|B|=\epsilon\gamma^2|B|\leq \epsilon\gamma|B'|$. 
It follows that for any $x\in G'\backslash Z'$, we have $|(x+P')\cap A'|<\epsilon\gamma|B'|$ or $|(x+P')\backslash  A'|<\epsilon\gamma|B'|$. 
Without loss of generality, we may assume $Z'\seq A'+P'$ while still maintaining the previous property and  $|Z'|<\epsilon|A'|$. By Proposition \ref{prop:TV}$(b)(ii)$, there are $C'\seq A'$ and $F'\seq A'\backslash Z'$ such that $|C'|,|F'|\leq\gamma\inv$, $A'\seq C'+B'$, and $A'\backslash Z'\seq F'+B'$. 

Let us summarize what has been proved. We have a proper coset progression $P'\seq G'$ of rank $s\leq O_{d,k,\epsilon}(1)$, and a set $Z'\seq A'+P'$, such that:
\begin{enumerate}[\hspace{5pt}$(1)$]
\item $|Z'|<\epsilon|A'|$ and $P'\seq A'-A'$,
\item if $x\in G\backslash Z'$ then either $|(x+P')\cap A'|<\epsilon\gamma|P'|$ or $|(x+P')\backslash A'|<\epsilon\gamma|P'|$,
\item there is $C'\seq A'$ such that $|C'|\leq \gamma\inv$ and $A'\seq C'+P'$, and
\item there is $F'\seq A'\backslash Z'$ such that $|F'|\leq \gamma\inv$ and $A'\backslash Z'\seq F'+P'$.
\end{enumerate}

Recall that we assumed $0\in A'$, and so the sets $A'$, $P'$, $A'+P'$, and $A'+2P'$ are all contained in $W':=3A'-2A'$. Let $W=3A-2A$. Then we have a well-defined Freiman $2$-isomorphism $\psi\colon W'\to W$ such that, given $a,b,c,x,y\in A'$,
\[
\psi(a+b+c-x-y)=\phi\inv(a)+\phi\inv(b)+\phi\inv(c)-\phi\inv(x)-\phi\inv(y).
\]
Note that if $x\in A'$, then $\psi(x)=\psi(x+0+0-0-0)=\phi\inv(x)$, and so $\psi(A')=\phi\inv(A')=A$. Moreover, since $\psi$ is a Freiman $2$-isomorphism and $\psi(0)=0$, it follows that for any $X,Y\seq W'$, if $X+Y\seq W'$ then $\psi(X+Y)=\psi(X)+\psi(Y)$, and if $X-Y\seq W'$ then $\psi(X-Y)=\psi(X)-\psi(Y)$. 

Let $P=\psi(P')$. By Proposition \ref{prop:TV}$(a)$,  $P$ is a proper coset progression of rank $s\leq O_{d,k,\epsilon}(1)$. Let $Z=\psi(Z')$. Then $|Z|=|Z'|<\epsilon|A'|=\epsilon|A|$. We show that $P$ and $Z$ satisfy conditions $(i)$, $(ii)$, and $(iii)$ in the statement of the theorem. Note first that $P\seq\psi(A'-A')=\psi(A')-\psi(A')=A-A$. Similarly, if we set $C=\psi(C')$ then using property $(3)$ above, together with the fact that $\psi$ is a bijection and $A=\psi(A')$, we have $|C|\leq \gamma\inv\leq O_{d,k,\epsilon}(1)$ and $C\seq A\seq C+P$. So we have $(i)$.

Next we show  $(iii)$ in a stronger form. Fix $g\in G\backslash Z$. If $g\not\in A+P$ then $(g+P)\cap A=\emptyset$, so assume $g\in A+P$. Thus we can write $g=\psi(x)$ for some $x\in (A'+P')\backslash Z'$. By property $(2)$, either $|(x+P')\cap A'|<\epsilon\gamma|P'|=\epsilon\gamma|P|$ or $|(x+P')\backslash A'|<\epsilon\gamma|P'|=\epsilon\gamma|P|$. Note that $x+P'\seq A'+2P'\seq W'$, and so $\psi((x+P')\cap A')=(g+P)\cap A$ and $\psi((x+P')\backslash A')=(g+P)\backslash A$. So either $|(g+P)\cap A|<\epsilon\gamma|P|$ or $|(g+P)\backslash A|<\epsilon\gamma|P|$. Since $\epsilon\gamma\leq\epsilon$, this yields  $(iii)$.

Finally, we prove $(ii)$. Let $F=\psi(F')$. So $|F|\leq\gamma\inv$ and $A\backslash Z\seq F+P$ by property $(4)$. Set $D=\{g\in F:|(g+P)\backslash A|<\epsilon\gamma|P|\}$, and note that if $g\in E:=F\backslash D$ then $|(g+P)\cap A|<\epsilon\gamma|P|$. Since $A\backslash Z\seq F+P$, it follows (as in the proof of Theorem \ref{thm:mainUP}) that
\[
A\smd (D+P)\seq Z\cup ((D+P)\backslash A)\cup((E+P)\cap A).
\] 
Therefore $|(A\smd(D+P))\backslash Z|\leq |F|\epsilon\gamma |P'|\leq \epsilon|P'|$. 
\end{proof}

\begin{remark}\label{rem:abSis}
As previously stated, in order to obtain explicit bounds in Theorem \ref{thm:NIPgen-ab} via the previous proof, one would need  a quantitative version of Theorem \ref{thm:CPTNIPab} for $f(u,v)=\epsilon(u/2v^{2v})^{2v}$. In particular, given explicit dependences $\delta(d,\alpha,\epsilon)$ and $r(d,\alpha,\epsilon)$ in Theorem \ref{thm:CPTNIPab}, the previous proof would yield $|C|\leq (2r^{2r}/\delta)^r$ in the statement of Theorem \ref{thm:NIPgen-ab}, where $\alpha=(100k)^{\nv 10k^2}$, $\delta=\delta(d,\alpha,\epsilon)$ and $r=r(d,\alpha,\epsilon)$.  The tools developed by Sisask in \cite{SisNIP} could be one avenue toward such a result. On the other hand, with Remark \ref{rem:qual} in mind, it is perhaps more reasonable to expect a version of Theorem \ref{thm:NIPgen-ab} with better dependence of $|C|$ on $\delta$ and $r$, but with the weaker statement $|A\smd (D+P)|<\epsilon\alpha |A|$ in condition $(ii)$. 
\end{remark}

\begin{proof}[\textnormal{\textbf{Proof of Theorem \ref{thm:NIPexp-ab}}}]
The strategy is the same as in Theorem \ref{thm:NIPgen-ab}, and so we will only sketch the important changes. We again fix an abelian group $G$ and a finite $d$-NIP set $A\seq G$ with $k$-doubling.  By Lemma \ref{lem:GR} and Proposition \ref{prop:FItame}, there is an abelian group $G'$ of exponent $r$, with $|G'|\leq c_r(k)|A|$, and a Freiman $6$-isomorphism $\phi'\colon A\to A'$ where $0\in A'\seq G'$ and $A'$ is $d$-NIP in $G'$. As before, we assume without loss of generality that $0\in A$ and $\phi(0)=0$.

Now fix $0<\delta\leq 1$ and let $\epsilon=\delta/c_r(k)$. Following the proof of \cite[Lemma 2.4]{AFZ}, there is a subgroup $H'\leq G'$ of index $n\leq \exp(O_r(d^8))\epsilon^{\nv d}$ such that $|(x+A')\smd A'|<\epsilon|G'|$ for all $x\in H'$. Let $D'=\{x\in G':|(x+H')\cap A'|>|H'|/2\}$. Then, as in \cite[Lemma 2.4]{AFZ}, it follows that $|A'\smd (D'+H')|<\epsilon|G'|\leq\delta|A|$. Using similar methods (which are detailed explicitly in \cite[Lemma 8.2]{CoBogo}), it also follows that there is some $Z'\seq A'+H'$, which is a union of cosets of $H'$ with $|Z'|<\epsilon^{1/2}|G'|\leq\delta^{1/2}|A|$, such that for all $x\in G'\backslash Z'$ either $|(x+H')\cap A'|<\epsilon^{1/4}|H'|$ or $|(x+H')\backslash A'|<\epsilon^{1/4}|H'|$. 

Since $[G':H']\leq n$, we can choose $C'\seq G'$ with $|C'|\leq n$ such that $A'\seq C'+H'$. By choosing $C'$ of minimal size, we have $(x+H')\cap A'\neq\emptyset$ for all $x\in C'$, and so after possibly changing coset representatives we may assume $C'\seq A'$. Similarly, we may assume without loss of generality that $D'\seq C'$. So the last thing we need before transferring back to $G$ is $H'\seq A'-A'$. To prove this, first note that $D'\neq\emptyset$ since $\delta\leq 1$. So fix $x_0\in D'$. For any $h\in H'$, we have $|(x_0+H')\cap (h+A')|=|(x_0+H')\cap A'|>|H'|/2$, which implies $(h+A')\cap A'\neq\emptyset$, i.e., $h\in A'-A'$.

Now one proceeds exactly as in the proof of Theorem \ref{thm:NIPgen-ab} to transfer the above situation back to $G$ via the Freiman $2$-isomorphism $\psi\colon 2A'-A'\to 2A-A$ induced by $\phi\inv$. (We can make do with $2A'-A'$ since $2H'=H'$.)
The only extra detail required is that $\psi(H')$ is a subgroup of $G$, which is easily verified. 
\end{proof}

\begin{remark}
Given the discussion of stable sets in Section \ref{sec:stable}, a natural question is whether a quantitative version of Theorem \ref{thm:MPW} can be obtained for abelian groups using a similar strategy with Freiman isomorphism, together with the (effective) arithmetic regularity results of Terry and Wolf \cite{TeWo} for stable subsets of finite abelian groups. We leave it as an exercise to verify that this is indeed the case. In particular, using similar arguments, one can show that if $A\seq G$ is $k$-stable and Freiman $s$-isomorphic to $A'\seq G'$, for some $s\geq 2$, then $A'$ is $k$-stable (in $G'$). On the other hand, this method is unnecessary in the stable case due to recent work of the first author \cite{CoQSAR} which directly gives a quantitative version of Theorem \ref{thm:MPW}, with improved bounds compared to what would be obtained from applying \cite{TeWo}. 
\end{remark}

\end{document}